\documentclass[11pt,reqno,a4paper]{amsart}
\usepackage[margin=0.75in]{geometry}
\usepackage[usenames]{color}
\usepackage{amsmath,pdfsync,verbatim,graphicx,epstopdf,enumerate}
\usepackage[colorlinks=true]{hyperref}
\usepackage{cancel,mathtools}
\usepackage{mathrsfs}
\usepackage[framemethod=tikz]{mdframed}
\hypersetup{allcolors=blue}
\mathtoolsset{showonlyrefs}
\allowdisplaybreaks
\numberwithin{equation}{section}
\newcommand{\I}{\mathrm{i}}

\newcommand{\lb}{\left(}

\newcommand{\rb}{\right)}
\newcommand{\PD}{\partial}

\newcommand{\Beq}{\begin{equation}}
	\newcommand{\Eeq}{\end{equation}}
\newcommand{\beq}{\begin{equation*}}
	\newcommand{\eeq}{\end{equation*}}
\newcommand{\bal}{\begin{align}}
	\newcommand{\eal}{\end{align}}

\newcommand{\g}{\gamma}

\usepackage{mathtools}

\usepackage[notref,notcite]{}
\usepackage[normalem]{ulem}

\newcommand{\B}{\beta}
\newcommand{\bp}{\begin{prob}}
	\newcommand{\ep}{\end{prob}}
\newcommand{\bpr}{\begin{proof}}
	\newcommand{\epr}{\end{proof}}
\renewcommand{\o}{\omega}

\newcommand{\bel}[1]{\begin{equation}\label{#1}}
	\newcommand{\ee}{\end{equation}}

\newcommand{\NT}{\negthinspace}
\usepackage{bm}
\newtheorem{theorem}{Theorem}[section]
\newtheorem{corollary}[theorem]{Corollary}

\newtheorem{lemma}[theorem]{Lemma}
\newtheorem{proposition}[theorem]{Proposition}

\newtheorem{question}{Question}

\theoremstyle{definition}

\newtheorem{remark}[theorem]{Remark}

\newcommand{\Rn}{\mathbb{R}^n}

\newcommand{\D}{\mathrm{d}}
\newcommand{\Lc}{\mathcal{L}}
\newcommand{\Rb}{\mathbb{R}}
\newcommand{\Dc}{\mathcal{D}}
\newcommand{\Nc}{\mathcal{N}}
\newcommand{\A}{\alpha}

\newcommand{\Cb}{\mathbb{C}}

\newcommand{\ve}{\varepsilon}

\newcommand{\Bb}{\mathbb{B}}
\usepackage[usenames]{color}
\usepackage{amsmath,pdfsync,verbatim,graphicx,epstopdf,enumerate}

\newcommand{\wt}{\widetilde}

\newcommand{\Fc}{\mathcal{F}}

\newcommand{\Pc}{\mathcal{P}}

\newcommand{\Rc}{\mathcal{R}}

\newcommand{\Sb}{\mathbb{S}}
\newcommand{\Sn}{\mathbb{S}^{n-1}}

\renewcommand{\o}{\omega}

\title[Range characterization of spherical mean transform]{A simple range characterization for spherical mean transform in odd dimensions and its applications}
\author[Agrawal,~ Ambartsoumian, ~Krishnan,~ Singhal]{Divyansh Agrawal$^\ddagger$, Gaik Ambartsoumian$^\dagger$, Venkateswaran P. Krishnan$^\ast$ and Nisha Singhal$^\ast$}

\address{$^\ddagger$ Department of Mathematics and Statistics, University of Jyväskylä, Finland
\newline
E-mail:{\tt \ divyansh.d.agrawal@jyu.fi, agrdiv01@gmail.com}
\newline
Orcid:{\tt\ 0009-0003-5125-0640}}
\address {$^{\ast}$ Centre for Applicable Mathematics, Tata Institute of Fundamental Research, Bangalore, India
\newline
E-mail:{\tt\  vkrishnan@tifrbng.res.in, nisha2020@tifrbng.res.in}
\newline
Orcid:{\tt\ 0000-0002-3430-0920, 0009-0006-3005-1986}}
\address{$^\dagger$ Department of Mathematics, The University of Texas at Arlington, Texas, USA
\newline
E-mail:{\tt \ gambarts@uta.edu}
\newline
Orcid:{\tt\ 0000-0002-1462-9964}}

\begin{document}
\begin{abstract}
    This article provides a novel and simple range description for the spherical mean transform of functions supported in the unit ball of an odd dimensional Euclidean space. The new description comprises of a set of symmetry relations between the values of certain differential operators acting on the coefficients of the spherical harmonics expansion of the function in the range of the transform.  As a central part of the proof of our main result, we derive a remarkable cross product identity for the spherical Bessel functions of the first and second kind, which may be of independent interest in the theory of special functions. Finally, as one application of the range characterization, we construct an explicit counterexample proving that unique continuation type results cannot hold for the spherical mean transform in odd dimensional spaces.
\end{abstract}


\maketitle

\section{Introduction}

The spherical mean transform (SMT), sometimes also called the spherical Radon transform, maps a function to its integrals over hyperspheres in $\Rn$. The study of this operator has a long history due to its relations to certain PDEs (wave equation, Euler-Poisson-Darboux equation) \cite{CH_Book, John-book, Rhee},  approximation theory and functional analysis \cite{agranovsky1996approximation, agranovsky1996injectivity}.
More recently, SMT and its inversion have been analyzed in connection with applications in tomography (see \cite{Kuchment-Kunyansky-TAT} and the references therein). 

The problem of determining a function from its averages over spheres is a formally over-determined problem and is usually studied in restricted settings, e.g., the centers are fixed on a hypersurface, or the radii are restricted \cite{Agranovsky-Kuchment-single_radius, Ambartsoumian2018, ref:AmbKuch, Cormack-Quinto}.  This article studies SMT of a function supported in the unit ball, and the centers of spheres of integration restricted to the boundary of the unit ball. In this setting, it is known that SMT is injective, and there are various formulas and algorithms for its inversion \cite{Ambartsoumian2018, Ambartsoumian-Zarrad-Lewis, ambartsoumian2015inversion,  ambartsoumian2014exterior, AER, aramyan2020recovering, Finch-Haltmeir-Rakesh_even-inversion, Finch-P-R, K, nguyen2009family, norton1980reconstruction, norton1981ultrasonic, R,  xu2002time}. An interesting feature of these inversion formulas is that they differ in odd and even dimensions and have local and non-local nature, respectively (see Section \ref{SRT-UCP}).  Recall that the solutions to the wave equation also show such features. For more detail, we refer the reader to the articles \cite{Finch-Rakesh-Survey, Kuchment-Kunyansky-TAT} and the references therein. 

In the context of investigating any generalized Radon transform and its inversion, it is desirable to have a description of the range of that operator. Such descriptions are valuable in analytical arguments dealing with various properties of these transforms. For example, the range characterization of the classical Radon transform in 2D was used to prove the non-uniqueness of the solution of the so-called interior problem in CT \cite{Natterer_book}. Furthermore, the range conditions (also often called data consistency conditions) can be useful in applications, since the measured (transform) data can be noisy or have missing parts, and the knowledge of the transform range may help with suppressing the noise or filling in the missing data (e.g. see \cite{anastasio2001comments, clarkson1999projections, mennessier1999attenuation, natterer1983exploiting, patch2004thermoacoustic}). Various range characterizations exist for SMT \cite{Agranovsky-Finch-Kuchment-range, Agranovsky-Kuchment-Quinto, AN, ref:AmbKuch-range, finch2006range, LVN}. However, the range conditions presented in the aforementioned articles are prohibitively complex to be used in constructive proofs, for example, when one needs to construct a function in the range of the transform with specified support constraints. In this work we derive a new characterization of the range of SMT in odd dimensions. Our range conditions are much simpler than those derived before, making them suitable for constructive proofs. Moreover, the proof of our main result produces a remarkable cross product identity for the spherical Bessel functions of the first and second kind. This identity may be of significant value as a standalone formula in the theory of special functions,  analogs of which we did not find in literature. Furthermore, it illuminates the structure of zeros of the Hankel transform of a function in the range of SMT, which was an essential and poorly understood component of the previously known range descriptions of SMT. Finally, as an application of our new range description, we use it to prove that the  unique continuation property (UCP) does not hold for the SMT in odd dimensions.  We also provide an alternative proof of the last statement without employing the range characterization.

The study of unique continuation property in the context of partial differential equations has a long and rich history \cite{Kotake-Narasimhan, Riesz, Tataru-survey}. More recently, UCP for integral transforms has attracted a lot of attention. Such results are possible in integral geometry due to their connections with non-local differential operators (for example, fractional powers of the Laplace operator). Unique continuation results for the X-ray transform of functions and vector-fields were studied in \cite{Keijo_partial_function, Keijo_partial_vector_field}, for tensor fields and for momentum transforms in \cite{Agrawal-Krishnan-Sahoo, ilmavirta2023unique} and for $d$-plane transforms in \cite{ucp-dplane,Covi-Monkkonen-Railo-UCP}. The unique continuation for $d$-plane transforms holds for odd $d$, i.e., when the surfaces of integration have odd dimensions. Our article proves that unique continuation does not hold when the surfaces of integration are spheres in an odd dimensional Euclidean space. Whether the UCP holds for SMT in even dimensions remains an open question. 

The rest of this article is organized as follows. In Section \ref{main-results}, we state our main results. We introduce relevant notation and give preliminaries in Section \ref{notations}. In particular, in Section \ref{BH} we recollect various formulas for Bessel functions and Hankel transforms used in the paper. In Section \ref{SRT}, we formally define the spherical mean transform and state a couple of known results about its inversion and range. In Section \ref{SRT-UCP}, we define the unique continuation property for SMT. Some basic mathematical results needed in the proofs are collected in Sections \ref{Egorychev} and \ref{aux-lemma}. Section \ref{proofs} is devoted to the proofs of the main theorems. In Section  \ref{range-Hankel}, we prove the range characterization for the SMT of radial functions. As part of the argument there, we derive an interesting and important cross product identity for Bessel functions of the first and second kind. Section \ref{sec:range-gen} deals with the range description of SMT in the general case. In Section  \ref{proof-CE} we construct a counterexample for UCP of SMT. Some combinatorial identities required for the proofs of the main results are given in Appendix A.

\subsection{Main results}\label{main-results} 
The main object of study in this paper is SMT. Informally,  as already mentioned above, SMT of a continuous function in $\Rn$ denotes the averages of the function over spheres with centers varying over $\Rn$ and positive radii. A formal dimension count gives that the SMT depends on $(n+1)$-variables, while the function itself depends on only $n$-variables. This, and certain applications in tomography, motivate restricting the centers of spheres to $(n-1)$-dimensional hypersurfaces, which makes the problem interesting as well as challenging.  

We will consider the case when the function is supported in $\Bb$ and the centers are fixed on $\Sn$. This can be easily generalized to balls and spheres of any radius by a simple scaling. For $f \in C_c^\infty (\Bb)$, the \textit{spherical mean transform} $\Rc$ is defined as 
\begin{equation*}
    \Rc f (p, t) = \frac{1}{\o_n} \int\limits_{\Sn} f(p+t\theta) \, \D S(\theta),
\end{equation*}
where $\o_n$ denotes the surface area of $\Sn$ and $\D S$ denotes the surface measure on it. Here we note that $p\in \Sb^{n-1}$ and $t\in (0,\infty)$. We caution the reader that some authors also define the above transform with weight $t^{n-1}$, in which case our results need to be modified accordingly. Due to the support restriction on $f$, $\Rc f (\cdot, t) = 0$ for $t \geq 2$. Thus, we have $\Rc: C_c^\infty(\Bb) \to C_c^\infty (\Sn \times (0,2))$.

Our first result gives a simple range characterization of $\Rc$ for radial functions.

\begin{theorem}[Range characterization for SMT of radial functions]\label{range}
    Let $\Bb$ denote the unit ball in $\mathbb{R}^n$ for an odd $n \geq 3$, and $k: = (n-3)/2$. A function $g \in C_c^\infty ((0,2))$ is representable as $g = \Rc f$ for a radial function $f \in C_c^\infty(\Bb)$ if and only if $h(t) \coloneqq t^{n-2} g(t)$ satisfies 
    \begin{align}\label{RC}
        [\Lc_k h](1-t) &= [\Lc_k h] (1+t), \quad \mbox{ for all } t \in [0,1],
    \end{align}
    where $\Lc_k$ is the linear differential operator of order $k$: 
   \Beq\label{RC2}
        \Lc_k = \sum\limits_{l=0}^k \frac{ (k+l)! }{(k-l)!\, l!\, 2^l} (1-t)^{k-l} D^{k-l}, \qquad D = \frac{1}{t} \frac{\D}{\D t},
    \Eeq
    and $[\Lc_k h](\cdot)$ denotes evaluation of the function $\Lc_k h$ at the given point.
\end{theorem}
\noindent Some special cases may be of particular interest which we highlight in the following remark. 
\begin{remark}
    In $\mathbb{R}^3$, $\Lc_0$ is the identity operator, therefore a function $g \in C_c^\infty ((0,2))$ is representable as $g = \Rc f$ for a radial function $f \in C_c^\infty(\Bb)$ if and only if $h(t) = t g(t)$ satisfies  
    \[
       h(1-t)=h(1+t),\mbox{ for all } t \in [0,1]. 
    \]     
    In $\mathbb{R}^5$, a function $g \in C_c^\infty ((0,2))$ is representable as $g = \Rc f$ for a radial function $f \in C_c^\infty(\Bb)$ if and only if $h(t)=t^3g(t)$ satisfies $[\Lc_1 h](1-t) = [\Lc_1 h] (1+t), \; \mbox{ for all } t \in [0,1]$, where
    \[
        \Lc_1 h\,(\tau)=\frac{1-\tau}{\tau}\,h'(\tau)+h(\tau).
    \]
It is easy to notice that as the dimension $n$ of the space grows, so does the order of the ordinary, linear,  differential operator $\Lc$ appearing in the symmetry relation \eqref{RC}.
\end{remark}

\begin{remark}
    The range condition \eqref{RC} can also be equivalently written as
    \begin{align*}
        [\Lc_k h](t) &= [\Lc_k h] (2-t), \quad \mbox{for all } t \in [0,1].
    \end{align*}
   In this form, the above condition is true for all $ t \in [0,2]$. 
\end{remark}
The above result for radial functions leads to a range characterization for SMT of arbitrary compactly supported smooth functions in the unit ball as follows.  To do so, we consider the spherical harmonics expansions of the functions $f$ and $g=\Rc f$.

We recall that the space of spherical harmonics functions of degree $m$ are the restrictions to the unit sphere $\Sb^{n-1}$ of homogeneous harmonic polynomials in $\Rb^n$ of degree $m$.  These are eigenfunctions of the spherical Laplacian $\Delta_{\Sb^{n-1}}$ with eigenvalue $-m(m+n-2)$. The dimension of the space of spherical harmonics of degree $m$ is $d_m$, where 
\[
d_{m}=\frac{(2m+n-2)(n+m-3)!}{m!(n-2)!}, \quad d_0 =1.
\]
For $m_1\neq m_2$, $Y_{m_1 l}$ and $Y_{m_2 l}$ are $L^2$ orthogonal functions on $\Sb^{n-1}$. We can choose the linearly independent set of spherical harmonics of a fixed degree $m$ to be orthonormal. We enumerate these spherical harmonics by $Y_{m,l}$ for $0\leq m<\infty$ and $1\leq l\leq d_m$. 

With these, we have 
\[
f(x)=\sum\limits_{m=0}^{\infty}\sum\limits_{l=1}^{d_{m}} f_{m,l}(|x|) Y_{m,l}(\frac{x}{|x|}),\quad 
f_{m,l}(r)=\int\limits_{\Sb^{n-1}} f(r\theta) \overline{Y}_{m,l}(\theta) \D \theta.
\]
Since $f\in C_c^{\infty}(\Bb)$, we have that $f_{m,l}\in C^{\infty}([0,1))$ with support strictly away from $1$.

Likewise, we expand $g=\Rc f$ into spherical harmonics: 
\[
g(\theta,t)=\sum\limits_{m=0}^{\infty}\sum\limits_{l=1}^{d_{m}} g_{m,l}(t) Y_{m,l}(\theta),
\]
with $g_{m,l}\in C_c^{\infty}((0,2))$.

\begin{theorem}[Range characterization - general case]\label{Thm1.4}
    Let $\Bb$ denote the unit ball in $\mathbb{R}^n$ for an odd $n \geq 3$, and $k: = (n-3)/2$. A function $g \in C_c^\infty (\Sb^{n-1}\times (0,2))$ is representable as $g = \Rc f$ for $f\in C_c^\infty(\Bb)$ if and only if  for each $(m,l), m\geq 0, 0\leq l\leq d_m$, $h_{m,l}(t)=t^{n-2}g_{m,l}(t)$ satisfies the following two conditions:
    \begin{itemize}
        \item there is a function $\phi_{m,l}\in C_c^{\infty}((0,2))$ such that 
        \Beq\label{GRC0}
            h_{m,l}(t)= D^{m} \phi_{m,l}(t), 
        \Eeq
        \item the function $\phi_{m,l}(t)$ satisfies 
        \begin{equation}\label{GRC}
           [\Lc_{m+k}\phi_{m,l}](1-t)= [\Lc_{m+k}\phi_{m,l}](1+t). 
        \end{equation}
    \end{itemize}
\end{theorem}

\textbf{Range characterization of the SMT using half of the radial data.} 
    It is well known that in the spherical geometry of data acquisition (i.e., when the centers $p$ of the integration spheres are restricted to the boundary of the unit ball containing the support of the function $f$), one can uniquely recover $f$ from $\Rc f(p,t)$ using only half of the radial data, i.e., when $p\in\mathbb{S}^{n-1}$ and $t\in(0,1)$ or $t\in(1,2)$ (e.g., see \cite{Ambartsoumian2018, Ambartsoumian-Zarrad-Lewis, ambartsoumian2015inversion}). In other words, the knowledge of $\Rc f(p,t)$ for $t\in(0,1)$ completely determines $\Rc f(p,t)$ for $t\in(1,2)$, and vice versa. Therefore, the existence of relations between the two halves of the data set is not surprising. The remarkable feature of relations \eqref{RC} and \eqref{GRC} is their simplicity, which has enabled the discovery and the proof of several important properties of SMT (see Sections \ref{proof-CE} and \ref{further} here, as well as \cite{AAKN2}). 

\begin{remark}    
    The symmetry relations \eqref{RC} and \eqref{GRC} can be easily interpreted as a range characterization of the SMT using half of the radial data. For example, let us discuss the case of radial functions.     The general (non-radial) case can be handled in a similar fashion. 
    
    Consider a function $g \in C^\infty ([0,1])$ such that $\operatorname{supp} g\subseteq [\varepsilon,1]$ for some $\varepsilon>0$, and $h(t) \coloneqq t^{n-2} g(t)$ for $t\in[0,1]$. Plugging $h$ into the left-hand side of equation \eqref{RC}, one gets an ODE for $h(t)$ when $t\in(1,2]$. Notice, that the support requirement and the symmetry relation imply that $h(t)$ and all its derivatives are zero when $t\in[2-\varepsilon, 2]$. Thus, the ODE described above has a unique solution $h(t)$, $t\in(1,2]$. 
    
    By construction, Theorem \ref{range} implies that $g(t)$, $t\in[0,1]$, is representable as $g = \Rc f$ for a radial function $f \in C_c^\infty(\Bb)$ if and only if the two branches of $h(t)$ patch smoothly at $t=1$ so that $h\in C_c^\infty((0,2)$). 

\end{remark}    
    
    In this context, we would like to mention a recent related work of Peter Kuchment and Leonid Kunyansky \cite{KK}. While the current work was in preparation, one of the authors presented a preliminary version of this work at Isaac Newton Institute for Mathematical Sciences (Cambridge, UK) in May 2023. Peter Kuchment and Leonid Kunyansky, who were at the talk, later informed us that they have derived a range characterization for the same SMT that we study, using information from radii  $0\leq r\leq 1$  in all dimensions; see \cite{KK}. Their description is implicit and requires expanding into a series the exterior Radon transform of the solution of an exterior initial/boundary value problem for a related wave equation.

\vspace{2mm}

Our next two results provide  counterexamples to UCP for SMT in odd dimensions (see Section \ref{SRT-UCP} for the precise definition).

\begin{theorem}[Counterexample to UCP for SMT in odd dimensions - symmetric case]\label{NUCP-1}
    Let $n \geq 3$ be odd, $\epsilon \in (0,1)$ and let $U = B_\epsilon(0) \coloneqq \{ x \in \Rn: |x| < \epsilon\}$. There exists a non-trivial function $f \in C_c^\infty(\Bb)$ such that $f$ vanishes in $U$ and $\Rc f(p, t)= 0$ for all $p \in \Sn$ and $ t \in ( 1-\epsilon, 1+\epsilon)$.
\end{theorem}

Note that the set $U$ here is taken to be a ball around the origin. One might wonder whether this is a special case due to  radial symmetry of the functions. However, this is not the case, and to disprove the unique continuation in full generality, we also have the following result.
\begin{corollary}[Counterexample to UCP for SMT in odd dimensions - general case]\label{NUCP-2}
    Let $n \geq 3$ be an odd integer and $U$ be an open set with $\overline{U} \subset \Bb$. There exists a non-trivial function $f \in C_c^\infty(\Bb) $ such that $f|_U =0$ and $\Rc f$ vanishes on all spheres passing through $U$.
\end{corollary}
This will be proved by using the symmetric case, Theorem \ref{NUCP-1}. \\

We finish this section with a short discussion of a result required to prove the sufficiency parts of Theorems \ref{range} and \ref{Thm1.4}.  This following result will be proved in Section \ref{range-Hankel}.

\begin{theorem}\label{Theorem:3.2-Prelim}
    Let $h\in C_c^{\infty}((0,2))$ satisfy the evenness condition \eqref{RC}. Then for any $\lambda >0$ and $k\in\mathbb{N}\cup \{0\}$, the following identity holds: 
    \begin{align}\label{special-formula-1}
\lb \int\limits_0^{\infty} j_{k+\frac{1}{2}}(\lambda t) th(t)\D t\rb y_{k+\frac{1}{2}}(\lambda)   =\lb \int\limits_0^{\infty} y_{k+\frac{1}{2}}(\lambda t) th(t) \D t\rb j_{k+\frac{1}{2}}(\lambda),
\end{align}
where $j_\alpha$ and $y_\alpha$ are the normalized (or spherical) Bessel functions of the first and second kind, respectively (see Section \ref{BH}).
\end{theorem}

Formula \eqref{special-formula-1} is remarkable for two reasons. First, it provides an infinite family (corresponding to different choices of $h$) of ``cross product'' identities for the spherical Bessel functions of the first and second kind, analogs of which we did not find in literature. Therefore, it may be valuable as a standalone result in the context of theory of special functions. Second, it illuminates the structure of the zeros of the Hankel transform of a function in the range of the SMT, which play an important role in the description of the range of that transform (see \cite{Agranovsky-Finch-Kuchment-range, Agranovsky-Kuchment-Quinto, ref:AmbKuch-range, finch2006range}).

\section{Notation and Preliminaries}\label{notations}

Let $n \geq 3$ be an odd integer of the form $n=2k+3, k \geq 0$  and $\Rn$ denote the $n$-dimensional Euclidean space. Let $\Bb$ denote the unit ball in $\Rn$ with its boundary denoted as $\Sn$. 

\subsection{Bessel functions and Hankel transform}\label{BH}
For $\alpha \in \Cb$ such that $\mathrm{Re} (\alpha) \geq 0$, the Bessel function of the first kind of order $\alpha$ is defined as (see for instance \cite{Trimeche}) 
\begin{align*}
    J_\alpha(x) &= \lb \frac{x}{2} \rb^\alpha \sum\limits_{i=0}^\infty \frac{(-1)^i (\frac{x}{2})^{2i}}{i! \Gamma(i+\alpha+1)}, \quad \mbox{for} \quad x \in (0,\infty).
\end{align*}
Bessel functions of order $\alpha$ are solutions of the second order differential equation
\begin{align*}
    \frac{\D^2 y}{\D x^2} + \frac{1}{x} \frac{\D y}{\D x} + \lb 1- \frac{\alpha^2}{x^2} \rb y &= 0,
\end{align*}
called the Bessel differential equation. 

Let us also define the normalized (or spherical) Bessel functions of the first kind. For $\alpha \in \Rb$ such that $\alpha > -1/2$, these are given as
\begin{align*}
    j_\alpha(x) &= \Gamma(\alpha+1)\lb \frac{2}{x} \rb^{\alpha} J_\alpha(x) \\
    &= \Gamma(\alpha+1) \sum\limits_{i=0}^\infty \frac{(-1)^i (\frac{x}{2})^{2i}}{i! \Gamma(i+\alpha+1)}.
\end{align*}
We are mostly interested in the case when $\alpha$ is half of an odd positive integer. In this case, $j_\alpha$ is also given by Rayleigh's formula \cite{Abramowitz-Stegun}
\begin{align}\label{Bessel1}
    j_{\alpha}(x) &= -\frac{(-2)^{\alpha+1/2}\Gamma(\alpha+1)}{\sqrt{\pi}} \lb \frac{1}{x} \frac{\D}{\D x}\rb^{\alpha-1/2} \lb \frac{\sin x}{x}\rb, \quad \mbox{when} \quad 2\alpha \in \{1,3, \dots\}.
\end{align}
We will also need the normalized Bessel function of the second kind of half integer order, which is defined  as
\begin{align}\label{Bessel2}
    y_\alpha(x) &= -\frac{(-2)^{\alpha+1/2}\Gamma(\alpha+1)}{\sqrt{\pi}} \lb \frac{1}{x} \frac{\D}{\D x}\rb^{\alpha-1/2} \lb \frac{\cos x}{x}\rb, \quad \mbox{when} \quad 2\alpha \in \{1,3, \dots\}.
\end{align}

\begin{remark}
    We caution the reader that the normalization of the Bessel functions is not standard. Our normalization differs from the one in \cite{Abramowitz-Stegun}. The Rayleigh's formula stated above has been modified accordingly. For a comprehensive study of Bessel functions, we refer the reader to the classical treatise of Watson \cite{Watson}.
\end{remark}

The Hankel (also called Fourier-Bessel or Fourier-Hankel) transform of order $\alpha$ is defined as
\begin{align*}
    \Fc_\alpha (g) (\lambda) &= \int\limits_0^\infty g(t) j_\alpha(\lambda t) t^{2\alpha+1} \, \D t.
    \intertext{Its inverse is given by}
    g(t) &= \frac{1}{2^{2\alpha} \Gamma^2(\alpha+1)} \int\limits_{0}^\infty \Fc_\alpha(g)(\lambda) j_\alpha(t\lambda) \lambda^{2\alpha+1} \, \D \lambda.
\end{align*}

\subsection{SMT - properties}\label{SRT}

In the setting of SMT discussed at the beginning of Section \ref{main-results}, the problem of inverting SMT has been considered by many authors, and explicit inversion formulas exist. Before stating the relevant inversion formulas, let us point out that when $f$ is a radial function, $\Rc f$ is independent of the center of integration. This can be seen by a simple application of the Funk-Hecke theorem. Since we will make extensive use of this result, we will record it here: 
\begin{theorem}[Funk-Hecke]\cite[Theorem 3]{Seeley}\label{Funk-Hecke}
If $\int\limits_{-1}^{1} |F(t)|(1-t^{2})^{\frac{n-3}{2}}\, \D t <\infty$, then for any $\eta\in \Sb^{n-1}$, 
\[
\int\limits_{\Sb^{n-1}}F\lb \langle \sigma,\eta\rangle \rb Y_{m,l}(\sigma)\, \D  S(\sigma) = \frac{\o_{n-1}}{C_{m}^{\frac{n-2}{2}}(1)}\lb\int\limits_{-1}^{1} F(t) C_{m}^{\frac{n-2}{2}}(t) (1-t^{2})^{\frac{n-3}{2}} \D t\rb Y_{m,l}(\eta),
\]
where $\lvert \Sb^{n-2}\rvert$ denotes the surface measure of the unit sphere in $\Rb^{n-1}$, $C_{m}^{\frac{n-2}{2}}(t)$ are the Gegenbauer polynomials and $Y_{m,l}$ are the spherical harmonics. 
\end{theorem} 
Using the above formula, we have, 
\Beq\label{radon-radial}
\begin{aligned}
    \Rc f(p,t) &= \frac{1}{\o_n} \int\limits_{\Sn} f(|p+t \theta|) \, \D S(\theta) \\
    &= \frac{1}{\o_n} \int\limits_{\Sn} f \lb \sqrt{1+t^2+ 2 t (p \cdot \theta)} \rb \, \D S(\theta)\\
&= \frac{\o_{n-1}}{\o_n} \int\limits_{-1}^1 f \lb \sqrt{1+t^2 +  2st} \rb (1-s^2)^{\frac{n-3}{2}} \, \D s.
\end{aligned}
\Eeq

The last equality is obtained by Fubini's theorem. Alternately one can use the Funk-Hecke theorem (Theorem \ref{Funk-Hecke}). Note that the right-hand side is independent of $p$.  This observation is not new and has been used to obtain inversion procedures for SMT. The above equation can be seen as a Volterra integral equation of the first kind with a weakly singular kernel, which can be modified into a Volterra integral equation of the second kind and then solved using Picard's method of successive iterations. This procedure is not specific to radial functions. The case of general functions can also be solved similarly by expansion into spherical harmonics, see \cite{Ambartsoumian-Zarrad-Lewis, ambartsoumian2015inversion, R, Salman_Article}. Due to the rotation invariance of SMT, the $n$-th term in the spherical harmonics expansion of $\Rc f$ depends only on the $n$-th term in the expansion of $f$ via a Volterra integral equation, which has a unique solution. It follows that if $\Rc f$ is independent of the centers of integration, then $f$ is necessarily a radial function. 

Let us now state an explicit inversion formula in odd dimensions which we use in our proofs.

\begin{theorem}\cite[Theorem~3]{Finch-P-R}\label{inversion-FPR}
    A smooth function $f \in C_c^\infty(\Bb)$ can be obtained from the knowledge of its spherical mean transform as follows:
    \begin{align}
        f(x) &= K(n) \lb \Nc^* \Dc^* \PD_t^2 t \Dc \Nc f \rb (x) \label{FPR-formula1} \\
        &= K(n) \lb \Nc^* \Dc^* \PD_t t \PD_t \Dc \Nc f \rb (x) \label{FPR-formula2}\\
        \notag &= K(n) \Delta_x \lb \Nc^* \Dc^* t \Dc \Nc f\rb (x),
    \end{align}
    where $K(n) = \frac{-\pi}{2 \Gamma (n/2)^2}$, and the various operators involved are given by
    \begin{align*}
    (\Nc f) (p,t) &= t^{n-2} (\Rc f)(p,t),
    \intertext{and for a function $G \in C_c^\infty(\Sn \times (0,2))$,}
    (\Dc G)(p,t) &= \lb \frac{1}{2t} \frac{\PD}{\PD t} \rb^{k} (G(p,t)), \\
    (\Nc^* G)(x) &= \frac{1}{\o_n} \int\limits_{\Sn} \frac{G(p, |p-x|)}{|p-x|} \D S(p), \\
    (\Dc^* G)(p,t) &= (-1)^k t \Dc \lb \frac{G(p,t)}{t} \rb.
\end{align*}
\end{theorem}

\begin{remark}
    The fact that $f$ is necessarily a radial function if $\Rc f$ is independent of the centers of integration can also be seen from the inversion formula above. If $\Rc f$ is independent of $p$, then so is $\Dc^* \PD_t^2 t \Dc \Nc$, and hence $\Nc^* (\Dc^* \PD_t^2 t \Dc \Nc)(x)$ depends only on $|x|$, again by an application of Funk-Hecke theorem. 
\end{remark}

Our proof of sufficiency is based on the following range characterization given in \cite{Agranovsky-Kuchment-Quinto}, where several equivalent conditions are given.  Interestingly, a follow-up paper of ours to the current work \cite{AAKN2} also shows the equivalence of the range characterization condition proven in this work and that of \cite{finch2006range}; see the comment immediately following Theorem 3 there, \emph{directly}.
\begin{theorem}\cite[Theorem~11]{Agranovsky-Kuchment-Quinto}\label{range-AKQ}
    Let $n > 1$ be an odd integer. A function $g \in C_c^\infty(\Sn \times (0,2))$ is representable as $\Rc f$ for some $f \in C_c^\infty(\Bb)$ if and only if for any $m$, the $m$th order spherical harmonic term $\widehat{g}_m (p,\lambda)$ of $\widehat{g}(p,\lambda)$ vanishes at non-zero zeros of the Bessel function $J_{m+n/2-1}(\lambda)$, where
    \begin{align*}
        \widehat{g}(p,\lambda) &= \Fc_{\frac{n-2}{2}} (g)(p,\lambda)
    \end{align*}
    is the Hankel transform of $g$ of order $\alpha = (n-2)/2$, for each fixed $p$.
\end{theorem}

\subsection{Unique continuation property for spherical mean transform}\label{SRT-UCP}
Let $\Pc$ denote any operator. For any open set $U$, if $\Pc u|_U = 0$ and $u|_U = 0$ implies that $u$ vanishes identically, then $\Pc$ is said to possess a \emph{unique continuation property}. Some examples of operators possessing UCP are fractional powers of the Laplacian, the normal operators of the X-ray and momentum ray transforms, normal operators of $d$-plane transforms (for $d$ odd), etc. In all these examples, the inversion formulas are non-local in nature.

Motivated by the results for X-ray and momentum ray transforms, we propose the following analog of UCP in the context of SMT:
\begin{question}[Unique continuation for spherical mean transform]
    Let $U \subset \Bb$ be an arbitrary open set. Let $f \in C_c^\infty(\Bb)$ be such that $f$ vanishes on $U$, and the spherical mean transform of $f$ vanishes on all spheres intersecting $U$. Does $f$ vanish identically?
\end{question}

A closer look at the inversion formula above reveals that in odd dimensions, the inversion formula for SMT is local in nature, that is, the value of the function $f$ at a point $x$ depends only on the spherical means of $f$ on spheres passing through a small neighbourhood of $x$. This observation suggests that a unique continuation result should not hold for $\Rc$ in odd dimensions. This is indeed true and is the content of Theorem \ref{NUCP-1} and Corollary~\ref{NUCP-2}.
\subsection{Technique of Egorychev to analyze combinatorial terms} \label{Egorychev}
In this work, we will deal with several combinatorial terms. Our analysis of these combinatorial terms is based on a technique pioneered by Egorychev \cite{Egorychev}.  For $n\geq k\geq 0$, we can write ${n\choose k}$ in terms of one of the following contour integrals: 
\begin{align} 
\label{Eg1}{n\choose k}& =\frac{1}{2\pi \I}\int\limits_{|z|=\ve} \frac{(1+z)^{n}}{z^{k+1}} \D z \\
\label{Eg2}&= \frac{1}{2\pi \I}\int\limits_{|z|=\ve} \frac{1}{(1-z)^{k+1}z^{n-k+1}} \D z. 
\end{align}
Note that with either of the contour integrals, if $n,k\geq 0$ with $n<k$, then ${n\choose k}=0$. Furthermore, ${0\choose 0}=1$ as well ${n\choose k}=0$ when $n\geq 0$ and $k<0$.  Some care is required when analyzing combinatorial terms when $n<0$. We make this precise in our analysis in Section \ref{proofs}.

\subsection{Some auxiliary lemmas}\label{aux-lemma}
In this subsection, we collect some basic mathematical results which will be used in the calculations. All these results are well known and are stated for the sake of completeness and easy reference. 

Let us begin by recalling the Fa\`a di Bruno formula, which is an identity relating the higher order derivatives of composition of two functions to the derivatives of the functions. This is a generalization of the usual chain rule  to higher order derivatives (see, for instance, \cite{Krantz-Parks_primer}).

\begin{lemma}[Fa\`a di Bruno formula]
    Let $F$ and $G$ be two smooth functions of a real variable. The derivatives of the composite function $F \circ G$ in terms of the derivatives of $F$ and $G$ are given as
    \begin{equation*}
        \frac{\D^p}{\D t^p} F(G(t)) = \sum\limits_{q=1}^p F^{(q)}(G(t)) B_{p,q} (G^{(1)}(t), \dots, G^{(p-q+1)}(t)),
    \end{equation*}
    where $B_{p,q}$ are the Bell polynomials given by
    \[
    B_{p,q}(x_1, \dots, x_{p-q+1}) = \sum \frac{p!}{j_1! \dots j_{p-q+1}!} \lb \frac{x_1}{1!}\rb^{j_1}\cdots  \lb \frac{x_{p-q+1}}{(p-q+1)!}\rb^{j_{p-q+1}},
    \]
with the sum taken over all non-negative sequences, $j_1,\cdots, j_{p-q+1}$ such that the following two conditions are satisfied: 
\begin{align*}
   & j_1+ j_2+\cdots + j_{p-q+1}= q,\\
   &j_1+2j_2+ \cdots + (p-q+1) j_{p-q+1}=p.
\end{align*}
\end{lemma}

We will be working with the operator $D$ defined as
\[
D  = \frac{1}{t} \frac{\D}{\D t}.
\]
Multiplying the standard chain rule by $\frac{1}{t}$, we see that the $D-$derivative of composition of two functions can then be re-written as
\[
D (F(G(t))) = F'(G(t)) \cdot DG(t),
\]
where $F'$ denotes the usual derivative of $F$. The following lemma is then an easy verification.
\begin{lemma}[Fa\`a di Bruno formula for the operator $D$]\label{FdB0}
    Let $F$ and $G$ be two smooth functions of 1-real variable. The $D$-derivatives of the composite function $F \circ G$ are given as
    \begin{equation*}
        D^p F(G(t)) = \sum\limits_{q=1}^p F^{(q)}(G(t)) B_{p,q} ((DG)(t), \dots, D^{(p-q+1)}G(t)).
    \end{equation*}
\end{lemma}

It can be quite difficult to work with the above formula in its full generality. However, for the case that we have at hand, applying Fa\`a di Bruno formula becomes much simpler. In our case, we have $D^j G = 0$ for $j \geq 3$, and the formula simplifies to

\begin{lemma}[Fa\`a di Bruno formula - special case]\label{FdB}
    Let $F$ and $G$ be two smooth functions of 1-real variable such that $D^jG = 0$ for $j \geq 3$. The following identity holds
    \begin{align*}
        D^p F(G(t)) &= \sum\limits_{q \geq p/2}^p  \frac{p!}{(2q-p)! (p-q)! 2^{p-q}} F^{(q)}(G(t)) \lb DG(t)\rb^{2q-p} \lb D^2G(t) \rb^{p-q}.
    \end{align*}
\end{lemma}

\begin{proof}
    Due to the existence of only two non-trivial $D$ derivatives of $G$, the Bell polynomials are subject to the following two conditions: 
    \[
        j_1+j_2=q, \quad j_1+2j_2=p.
    \]
    Solving this gives the following unique solution: $j_2=p-q$ and $j_1=2q-p$. Since $j_1\geq 0$, we have the additional requirement that $q\geq p/2$. With all these considerations, we arrive at the required formula for $D^p(F(G(t)))$.
\end{proof}

Finally, let us record the expression for repeated integration by parts with the operator $D$.

\begin{lemma}\label{lemma-IBP}
    For two smooth functions $F$ and $G$, the following identity holds:
    \begin{align}
        \int\limits_{a}^b \PD_t D^k F \cdot G \, \D t &= \left [ \sum\limits_{l=0}^{k-1} (-1)^l D^{k-l} F \cdot D^{l} G \right ]_{t=a}^{b} + (-1)^k \int\limits_a^b \PD_t F \cdot D^{k} G \, \D t, 
    \end{align}
    where the sum is interpreted as empty for $k=0$.
\end{lemma}
\noindent The proof is straightforward and hence omitted.

\section{Proof of main results}\label{proofs}

\subsection{Range characterization for radial functions}\label{range-Hankel}
In this section, we give the proof of Theorem~\ref{range}. Let us begin by giving our motivation briefly using the figure below, and by considering the case of 3-dimensions, where the necessary condition is fairly straightforward to observe. When the function $f$ possesses radial symmetry, some relation between $\Rc f$ at points $1 \pm t$ is expected, as the figure below suggests.
\begin{figure}[ht]
    \centering  \includegraphics[scale=0.6]{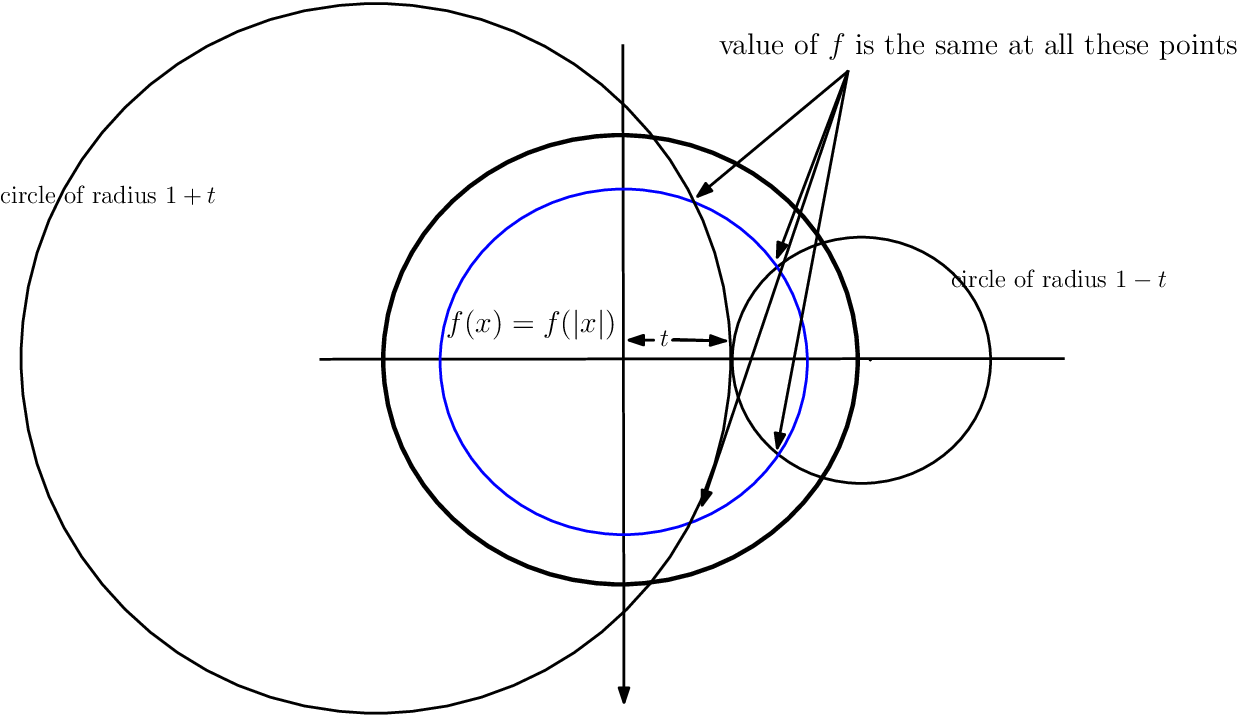}
    \caption{Relation between SMT of a radial function at radii $(1+t)$ and $(1-t)$ for $0<t<1$.}
    \label{figure}
\end{figure}
Notice that both the spheres (of radii $1 \pm t$) pass through points having the same values of $f$. Let us now consider the case of $3$-dimensions. For $t \in (0,2)$, we have  
\begin{align*}
    \Rc f(t) &= \frac{2\pi}{4\pi} \int\limits_{-1}^1 f \lb \sqrt{1+t^2 +  2st} \rb  \, \D s.
\end{align*}
Consider the change of variables $u = \sqrt{1+t^2 +  2st}$ to obtain 
\begin{align*}
    \Rc f (t) &= \frac{1}{2t}  \int\limits_{|1-t|}^{1+t} u f(u) \, \D u \\
    &= \frac{1}{2t}  \int\limits_{|1-t|}^{1} u f(u) \, \D u,
\end{align*}
since $f$ vanishes outside the unit ball and it follows that the function $t \Rc f (t)$ satisfies
\begin{align*}
    [t \Rc f ](1-t) & = [t \Rc f ](1+t) \mbox{ for } t\in[0,1],
    \intertext{or equivalently}
    [t \Rc f ](t) & = [t \Rc f ](2-t) \mbox{ for } t\in[0,1].
\end{align*}
This relation also suggests working with $t^{n-2} \Rc f$ instead of $\Rc f$. 

\subsubsection{Preliminary calculation in 5-D}
In order to see what condition to expect, let us consider the case of spherical mean transform in 5-dimensions, which is computationally the first non-trivial case. Let $f$ be a smooth radial function supported in the unit ball in $\Rb^5$, that is, $f(x)=\wt{f}(|x|)$ for a smooth compactly supported function $\wt{f}$ on $[0,\infty)$.  In order to avoid proliferation of new notation, we use the same $f$ to denote the function of one variable associated to $f$. Then 
\Beq
\begin{aligned} \label{9.1}
\Rc f(p,t)&=\frac{1}{\o_{5}}\int\limits_{\Sb^4} f(p+ t\theta) \D \theta\\
&=\frac{1}{\o_{5}}\int\limits_{\Sb^4} f(\sqrt{1+t^2 + 2 tp\cdot\theta}) \D \theta.
\end{aligned}
\Eeq
Applying Funk-Hecke theorem, we get,
\Beq\label{19.4}
g(t)=\frac{\o_4}{\o_5}\int\limits_{-1}^{1} f(\sqrt{1+t^2 -2st}) (1-s^2) \D s.
\Eeq
In the equation above, $\o_4$ and $\o_5$ denote the area of the unit sphere in $\Rb^4$ and $\Rb^5$, respectively. 
For $t > 0$, making the change of variable, $u=\sqrt{1+t^2 -2 st}$, we get, 
\begin{align*}
g(t)&= \frac{\o_4}{t\o_5}\int\limits_{|1-t|}^{1} f(u) u \lb 1-\lb \frac{1+t^2-u^2}{2t}\rb^2\rb \D u\\
&= \frac{\o_4}{4t^{3}\o_5}\int\limits_{|1-t|}^{1} f(u) u \lb 4t^2-\lb 1+t^2-u^2\rb^2\rb \D u.
\end{align*}

Let us denote 
\[
h(t)=t^{3}g(t) \mbox{ and } C=\frac{\o_4}{4 \o_5}.
\]
Then 
\begin{align}
\notag h(t)& =C\int\limits_{|1-t|}^{1} f(u) u \lb 4t^2-(1+t^2-u^2)^2\rb \D u\\
\notag &= C\int\limits_{|1-t|}^{1} f(u) u\lb (1+t)^2-u^2\rb (u^2-(t-1)^2) \D u.
\end{align}
We let $0<t<1$. Then 
\begin{align}
  \label{9.7}  h(t)=C\int\limits_{1-t}^{1} f(u) u\lb (1+t)^2-u^2\rb (u^2-(t-1)^2) \D u.
\end{align}We replace $t$ by $2-t$ in the above expression. We get, 
\begin{align}\label{9.8}
    h(2-t)=C\int\limits_{1-t}^{1} f(u) u ((3-t)^2-u^2)(u^2-(t-1)^2) \D u.
\end{align}
For simplicity of notation, let us denote 
\[
\A=C\int\limits_{1-t}^{1} f(u) u \D u,\quad \B=C\int\limits_{1-t}^{1} f(u) u^3 \D u, \quad \g=C\int\limits_{1-t}^{1} f(u) u^5 \D u.
\]
In this notation
\begin{align}\label{9.11}
    h(t)=-(t^2-1)^2 \A +2(1+t^2) \B -\g.
\end{align}
\begin{align}\label{9.12}
    h(2-t)=-((3-t)(1-t))^2 \A +2(1+(2-t)^2) \B - \g.
\end{align}
Eliminating $\g$, we have 
\Beq\label{9.12A}
h(t)+(t^2-1)^2\A -2(1+t^2)\B= h(2-t)+((3-t)(1-t))^2 \A -2(1+(2-t)^2) \B.
\Eeq
Therefore, from \eqref{9.7} and \eqref{9.8}, it is enough to find expressions for $\A$ and $\B$. Differentiating these expressions, we get, 
\begin{align}\label{9.13}
    h'(t)=-4t(t^2-1)\A +4t \B.
\end{align}
\begin{align}\label{9.14}
    h'(2-t)=4(t-1)(t-2)(t-3) \A -4(t-2)\B.
\end{align}
Note that those terms which involve the derivative of the integral add to $0$.
Solving \eqref{9.13} and \eqref{9.14}, we get, 
\begin{align}\label{EqSect1:15}
    \A = -\frac{(t-2)h'(t) + th'(2-t)}{16 t(t-1)(t-2)}.
\end{align}
\begin{align}\label{EqSect1:16}
    \B=-\frac{(t-2)(t-3)h'(t) + t(t+1)h'(2-t)}{16t(t-2)}.
\end{align}
Substituting this back into \eqref{9.12A}, we then get, 
\begin{align}
    h(t) + \frac{(1-t)}{t} h'(t)=h(2-t) + \frac{(1-t)}{(t-2)} h'(2-t) \mbox{ for all } t\in(0,1).
\end{align}
In the notation of $D$ operator, we then get, 
\begin{align}
    h(t)+(1-t)[Dh](t)=h(2-t)-(1-t)[Dh](2-t) \mbox{ for all } t\in (0,1).
\end{align}
By continuity, we also have 
\begin{align}
    h(t)+(1-t)[Dh](t)=h(2-t)-(1-t)[Dh](2-t) \mbox{ for all } t\in [0,1].
\end{align}
This can be rewritten in the final form as: 
\begin{align}\label{EqSect1:20}
    h(t)-h(2-t) + (1-t)\lb [Dh](t) +[Dh](2-t)\rb =0 \mbox{ for all } t\in [0,1].
\end{align}
Note that due to the smoothness condition on $h$, the expression above is well-defined for $t=1$ as well.

\subsubsection{Proof of necessity in Theorem \ref{range}}
Our goal next is to generalize the above approach for odd dimensional spherical Radon transform set-up. The strategy, as in this specific example, is to eliminate integral expressions involving $f$. We also make the following observations: 
\begin{itemize}
    \item In the general odd dimensional set-up, we can take up to $k^{\mathrm{th}}$ order $D$ derivatives, where $k=(n-3)/2$, and all such derivatives pass through the integral. In other words, the derivatives of the limits in the integral have no contribution up to the $k^{\mathrm{th}}$ order.
    \item Based on the calculations done for the 5D-case, we consider coefficients of $D$ derivatives as powers of $(1-t)$ multiplied by suitable constants. As in \eqref{EqSect1:20}, these are subtracted when evaluated at $t$ and $(2-t)$ for even order $D$ derivatives and added for odd order $D$ derivatives and set to $0$ to determine the coefficients.
    \end{itemize}
We carry out this program for the general odd dimensional case now. We should mention here that while the computations done for the  5D case serve as a motivation for our approach below, it is very difficult to generalize it to higher dimensional cases, since the solution to the problem relies on the explicit inversion of a matrix. Nevertheless, finding the correct combination of derivatives leads to a positive answer as we show below.  The 3D case is trivial, and the 5D computations done above can be recast as follows: Let us start with the expression for $h(t)$: 
\begin{align*}
    h(t)&= \frac{\o_4}{4\o_5}\int\limits_{|1-t|}^{1} f(u)u \lb 4t^2-(1+t^2-u^2)^{2}\rb \D u\\
    &= \frac{\o_4}{4\o_5}\int\limits_{|1-t|}^{1} f(u)u\lb 2(u^2+1)t^2 -t^4-(1-u^2)^2\rb \D u.
\end{align*}
Let
\[
P(t,u)= 2(u^2+1)t^2 -t^4-(1-u^2)^2.
\]
It is straightforward to check that 
\[
\lb P(t,u)-P(2-t,u)\rb +(1-t)\lb [DP](t,u)+[DP](2-t,u)\rb \equiv 0.
\]
This then gives that 
\[
\lb h(t)-h(2-t)\rb +(1-t)\lb [Dh](t)+[Dh](2-t)\rb \equiv 0.
\]
This is exactly what we derived earlier using a slightly different approach. Nevertheless, this serves as a motivation for what follows.

For the analysis in the general case, we require a few computations and several combinatorial results, which we state and prove below.  

We make the following convention while analyzing combinatorial terms: For $r\geq 0$,  ${r\choose s}=0$ whenever $r<s$ or $s<0$.

\begin{proposition}\label{NecessityProp}
    Let $Q(t,u)=2(u^2+1)t^2-t^4-(1-u^2)^2$ and $P(t,u)=\lb Q(t,u)\rb^{k}$, where $k=\frac{n-3}{2}$. Then for $0\leq p\leq k$,  there exist coefficients $C(k,p)$ with $C(k,k)=1$ such that 
    \[
    \sum\limits_{p=0}^{k} C(k,p) (1-t)^{p}\lb [D^p P](t,u)+(-1)^{p+1}[D^p P](2-t,u)\rb =0 \mbox{ for all }0\leq t\leq 1 \mbox{ and for all }  u \in (0,2).
    \]
    In fact, the coefficients are $C(k,p)=\frac{(2k-p)!}{p!2^{k-p}(k-p)!}$.
\end{proposition}
\bpr
We first find an expression for the higher order $D$ derivatives of $P(t,u)$ using the special case of Fa\`a di Bruno formula, Lemma \ref{FdB}.
We observe that  
\[
[DQ](t,u)=4(u^2+1-t^2) = 4 \lb \frac{Q(2-t,u)-Q(t,u)}{8(1-t)} +2(1-t)\rb,\quad [D^2 Q](t,u)=-8,
\]
and 
$[D^{p}Q](t,u)=0$ for $p\geq 3$. Then by Lemma \ref{FdB}, 
\begin{align*}
[D^{p}P](t,u)= \sum\limits_{q\geq p/2}^{p} \frac{k!}{(k-q)!}& \lb Q(t,u)\rb^{k-q}\frac{p!}{(2q-p)! (p-q)! 2^{p-q}}\\
&\times \lb \frac{Q(2-t,u)-Q(t,u)}{2(1-t)} +8(1-t)\rb^{2q-p}(-8)^{p-q}.
\end{align*}
We rewrite this as 
\begin{align}\label{Eq16.5}
    [D^{p}P](t,u)&=\sum\limits_{q\geq p/2}^{p} \frac{K(p,q)}{(1-t)^{2q-p}} Q(t,u)^{k-q}(Q(2-t,u)-Q(t,u) + 16(1-t)^2)^{2q-p},
\end{align}
with 
\[
K(p,q)= \frac{k! p! (-4)^{p-q}}{(k-q)!(2q-p)!(p-q)!2^{2q-p}}.
\]

Since we are only interested in derivatives in the $t$ variable, we will suppress the dependence of $P,Q$ and their derivatives on $u$, and simply write $P(t), Q(t)$, etc. 
Expanding the right-hand side of \eqref{Eq16.5} by the binomial theorem, 
\begin{align*}
    [D^{p} P](t)
    &=\sum\limits_{q\geq p/2}^{p} \sum\limits_{r=0}^{2q-p}\frac{16^r K(p,q)}{(1-t)^{2q-p}}(1-t)^{2r} \binom{2q-p}{r} Q(t)^{k-q} (Q(2-t)-Q(t))^{2q-p-r}.
\end{align*}
Replacing $t$ by $2-t$, 
\begin{align*}
    [D^{p} P](2-t)\NT=\NT\NT\NT\NT\sum\limits_{q\geq p/2}^{p} \sum\limits_{r=0}^{2q-p}\frac{(-1)^{2q-p-r}16^r K(p,q)}{(-1)^{2q-p}(1-t)^{2q-p}}(1-t)^{2r} \binom{2q-p}{r} Q(2-t)^{k-q} (Q(2-t)-Q(t))^{2q-p-r}.
\end{align*}
Therefore, 
\begin{align*}
   (1-t)^{p} \Big{(} [D^{p} P](t)+(-1)^{p+1} & [D^{p} P](2-t)\Big{)}=\sum\limits_{q\geq p/2}^{p} \sum\limits_{r=0}^{2q-p} K(p,q) 16^r (1-t)^{2p-2q+2r} \binom{2q-p}{r} \\
   &\times (Q(2-t)-Q(t))^{2q-p-r}\Bigg{\{} Q(t)^{k-q} +(-1)^{p-r+1} Q(2-t)^{k-q}\Bigg{\}}.
\end{align*}
We want to find coefficients $\{C(k,p)\}$ for $0\leq p\leq k$ with $C(k,k)=1$ such that
\begin{align*}
   \notag &\sum\limits_{p=0}^{k} \sum\limits_{q\geq p/2}^{p} \sum\limits_{r=0}^{2q-p}C(k,p) K(p,q) (-1)^{2q-p-r}16^r (1-t)^{2p-2q+2r}{2q-p\choose r}\\
   \notag \times&(Q(t)-Q(2-t))^{2q-p-r}\Bigg{\{} Q(t)^{k-q} +(-1)^{p-r+1} Q(2-t)^{k-q}\Bigg{\}}\equiv 0.
\end{align*}
Simplifying the constants in the equality above,
\Beq\label{Eq16.16}
\begin{aligned}
    & \sum\limits_{p=0}^{k} \sum\limits_{q\geq p/2}^{p} \sum\limits_{r=0}^{2q-p}C(k,p) \frac{k! p!(-1)^{q-r} 2^{3p-4q+4r}}{(k-q)!(p-q)! r! (2q-p-r)!} (1-t)^{2(p-q+r)}  \\
    &\times (Q(t)-Q(2-t))^{2q-p-r}\Big{\{} Q(t)^{k-q} +(-1)^{p-r+1} Q(2-t)^{k-q}\Big{\}}=0.
\end{aligned}
\Eeq
 Our strategy for determining the coefficients $C(k,p)$ is to set the sum of terms corresponding to a fixed power of $(1-t)$, say  $(1-t)^{2(k-l)}$, equal to $0$ for each $l\geq 0$.

We first claim that a term of the form $(1-t)^{2k}$ does not appear in the above expansion. This can be seen as follows: To obtain the term $(1-t)^{2(p-q+r)}=(1-t)^{2k}$, we must have $p-q+r=k$. 
Since $r \leq 2q-p$, we have $p-q+r \leq p$. Therefore, $(1-t)^{2k}$ does not arise for $p < k$. Hence, we must necessarily have $p=k$.  This then gives $r=q$, which then implies $p\leq q$. But $q\leq p$ always, and so we must have $q=p=k$. This forces $r=k$. Due to the presence of $(-1)^{p-r+1}$ in the term above, we get that $(1-t)^{2k}$ term does not appear in the expansion above. This explains the choice we make for the coefficients $C(k,k)$. We choose this to be $1$. 

We next show the validity of the formula for $C(k,k-1)$.  Toward this, we next show that a term involving $(1-t)^{2k-2}$ appears exactly twice in the term involving $C(k,k)$ and once in the term involving $C(k,k-1)$. First, consider $p=k$. Then we have to consider $p-q+r=k-1$, and since $p=k$, we have $q-r=1$, which then implies that $k-1\leq q$. Hence the two choices of $q$ that are possible are $q=k-1$ and $q=k$. If $q=k-1$, then $r=k-2$, and if $q=k$, then $r=k-1$. If we consider $p=k-1$, then exactly the same argument as in the previous paragraph leads to $r=q=p=k-1$. Hence only one choice is possible. Next let $p=k-2$. We then have $r-q=1$, and following the same arguments as above, we get that $q\geq k-1$, which is impossible since $q\leq p=k-2$. A similar argument follows for all $p<k-2$. Hence we have established that there are exactly three terms. 

Summarizing the contents of the above paragraph, there are exactly 3 terms in the above expansion involving $(1-t)^{2k-2}$. They correspond to the following triples: 
\begin{itemize}
\item $(p,q,r)=(k,k-1,k-2)$, 
\item $(p,q,r)=(k,k,k-1)$,
\item $(p,q,r)=(k-1,k-1,k-1)$. 
\end{itemize} We have to be careful with terms involving the case when $q=k$, since in this case, $Q(t)^{k-q} +(-1)^{p-r+1}Q(2-t)^{k-q}$ is either $0$ or $2$. In the case at hand, it is $2$, since $p-r=1$, and this extra factor of $2$ was taken into account in the second summand in \eqref{Eq:3.20AA} below.  
Setting the sum of terms involving $(1-t)^{2k-2}$ to $0$, we get  
\begin{align}\label{Eq:3.20AA}
    -C(k,k)\Bigg{\{}\frac{4 k!^2  8^{k-2}}{(k-2)!} + k! 8^{k-1} k\Bigg{\}} + C(k,k-1) k! 8^{k-1} =0.
\end{align}
Since $C(k,k)=1$, we find that $C(k,k-1)=k(k+1)/2$. 
 We assume by induction that $C(k,k-s)=\frac{(k+s)!}{(k-s)!2^{s} s!}$ for all $0\leq s\leq l-1$. 
 Our goal is to determine $C(k,k-l)$. Let us assume $l$ is odd; the proof for the even case is similar. For ease of notation, from now on, we let $A=Q(t)$ and $B=Q(2-t)$.

 With $l$ fixed, the maximum possible choices of triples $(p,q,r)$ such that the terms corresponding to these triples give the power $(1-t^2)^{2(k-l)}$ are: 
\begin{enumerate}
    \item[0.] $(k,k-l,k-2l), (k,k-l+1, k-2l+1),\cdots, (k,k,k-l)$
    \item[1.] $(k-1,k-l,k-2l+1), (k-1,k-l+1, k-2l+2),\cdots, (k-1,k-1,k-l)$
    \item[]\hspace{2in} $\vdots$
    \item[$l$.] $(k-l,k-l,k-l)$.
\end{enumerate}
The total number of terms above is $\frac{(l+1)(l+2)}{2}$.

The terms corresponding to the triples from (0) above are: 
\begin{align*}
-C(k,k) (k!)^2 2^{3k-4l}\sum\limits_{s=0}^{l} \frac{(A-B)^{s}\lb A^{l-s}+(-1)^{2l+1-s}B^{l-s}\rb}{((l-s)!)^2 (k-2l+s)!s!}.
\end{align*}
The terms corresponding to the triples $(p,q,r)$ from (1) above are:
\begin{align*}
    C(k,k-1) k!(k-1)!2^{3k-4l+1}\sum\limits_{s=0}^{l-1} \frac{(A-B)^{s}(A^{l-s}+(-1)^{2l-s-1}B^{l-s})}{(l-s)!(l-s-1)! (k-2l+1+s)!s!}.
\end{align*}
The terms corresponding to triples $(p,q,r)$  with $p=k-2$  are: 
\begin{align*}
    -C(k,k-2) k!(k-2)! 2^{3k-4l+2}\sum\limits_{s=0}^{l-2}\frac{(A-B)^{s}(A^{l-s}+(-1)^{2l-s-3}B^{l-s})}{(l-s)!(l-s-2)!(k-2l+2+s)! s!}.
\end{align*}
Continuing in this fashion and summing up all the terms corresponding to $(1-t)^{2(k-l)}$ and setting it to $0$, we have 
\begin{align*}
    k!\sum\limits_{m=0}^{l} (-1)^{l-m} C(k,k-m) (k-m)! 2^{3k-4l+m} \sum\limits_{s=0}^{l-m} \frac{(A-B)^{s}(A^{l-s}-(-1)^{2l-s} B^{l-s})}{(l-s)! (l-s-m)! (k-2l+s+m)! s!}=0.
\end{align*}
We ignore $k!$ and $2^{3k-4l}$ in the above expression from now on. Interchanging the order of summation, we get, 
\begin{align}\label{Eqs3.20}
    \sum\limits_{s=0}^{l}\sum\limits_{m=0}^{l-s} (-1)^{l-m} C(k,k-m) (k-m)! 2^{m} \frac{(A-B)^{s}(A^{l-s}-(-1)^{s} B^{l-s})}{(l-s)! (l-s-m)! (k-2l+s+m)! s!}=0.
\end{align}
Let us split \eqref{Eqs3.20} as 
\begin{align*}
   \notag  &\sum\limits_{s=1}^{l-1}\sum\limits_{m=1}^{l-s}(-1)^{l-m} C(k,k-m) (k-m)! 2^{m} \frac{(A-B)^{s}(A^{l-s}-(-1)^{s} B^{l-s})}{(l-s)! (l-s-m)! (k-2l+s+m)! s!}\\
    \notag &+ \sum\limits_{m=1}^{l-1} \frac{(-1)^{l-m} C(k,k-m) (k-m)! 2^m (A^l-B^l)}{l! (l-m)! (k-2l +m)!}-2(A-B)^{l}{k\choose l}\\
    &- \sum\limits_{s=1}^{l-1} \frac{k!(A-B)^{s}(A^{l-s}-(-1)^{s} B^{l-s})}{((l-s)!)^2(k-2l+s)!s!} + \frac{C(k,k-l) 2^{l}(A^l-B^l)}{l!} - \frac{k!(A^l-B^l)}{(l!)^2(k-2l)!}=0. 
\end{align*}
Using the fact that $C(k,k-m)=\frac{(k+m)!}{(k-m)! 2^{m} m!}$ for $0\leq m<l$, we get, 
\begin{align*}
 &\sum\limits_{s=1}^{l-1}\sum\limits_{m=1}^{l-s}(-1)^{l-m} \frac{(k+m)!(A-B)^{s}(A^{l-s}-(-1)^{s} B^{l-s})}{m!(l-s)!(l-s-m)!(k-2l+s+m)! s!} - \sum\limits_{m=1}^{l-1} (-1)^{m} \frac{(k+m)!(A^l-B^l)}{l! m!(l-m)!(k-2l+m)!}\\
   \notag  &-2(A-B)^{l}{k\choose l} - \sum\limits_{s=1}^{l-1} \frac{k!(A-B)^{s}(A^{l-s}-(-1)^{s} B^{l-s})}{((l-s)!)^2(k-2l+s)!s!} + \frac{C(k,k-l) 2^{l}(A^l-B^l)}{l!} - \frac{k!(A^l-B^l)}{(l!)^2(k-2l)!}=0.
\end{align*}
We can write this as 
\Beq\label{Eqs1.46}
\begin{aligned}
   & -\sum\limits_{s=1}^{l-1} \frac{(2l-s)!(A-B)^{s}(A^{l-s}-(-1)^{s} B^{l-s})}{((l-s)!)^2 s!}\sum\limits_{m=1}^{l-s} (-1)^m {k+m\choose 2l-s}{l-s \choose m}\\
   &-\frac{(2l)!(A^l-B^l)}{(l!)^2}\sum\limits_{m=1}^{l-1}(-1)^m {k+m\choose 2l}{l\choose m}
  -2(A-B)^l{k\choose l}\\
   &-\sum\limits_{s=1}^{l-1}{k\choose 2l-s}{2l-s\choose l}{l\choose s} (A-B)^{s}(A^{l-s}-(-1)^{s} B^{l-s})\\ &  + \frac{C(k,k-l) 2^{l}(A^l-B^l)}{l!} - \frac{k!(A^l-B^l)}{(l!)^2(k-2l)!}=0.
\end{aligned}
\Eeq
 Using (a) from Lemma \ref{CombinatorialLemmas} and noting that $l$ is assumed to be odd, we can simplify the above equality as 
\Beq\label{Eqs1.46AA}
\begin{aligned}
   & \sum\limits_{s=1}^{l-1} \frac{(2l-s)!(A-B)^{s}(A^{l-s}-(-1)^{s} B^{l-s})}{((l-s)!)^2 s!} \lb (-1)^{s}{k\choose l}+{k\choose 2l-s}\rb \\
   &+\frac{(2l)!(A^l-B^l)}{(l!)^2}{k\choose l}+\frac{(2l)!(A^l-B^l)}{(l!)^2}{k\choose 2l}- \frac{(2l)!(A^l-B^l)}{(l!)^2}{k+l\choose 2l}
  -2(A-B)^l{k\choose l}\\
   &-\sum\limits_{s=1}^{l-1}{k\choose 2l-s}{2l-s\choose l}{l\choose s} (A-B)^{s}(A^{l-s}-(-1)^{s} B^{l-s})\\ &  + \frac{C(k,k-l) 2^{l}(A^l-B^l)}{l!} - \frac{k!(A^l-B^l)}{(l!)^2(k-2l)!}=0.
\end{aligned}
\Eeq
We write 
\begin{align*}
    \frac{(2l-s)!}{((l-s)!)^2 s!}= {2l-s \choose l-s}{l\choose s}= {2l-s \choose l}{l\choose s}.
\end{align*}
Using (b) from Lemma \ref{CombinatorialLemmas}, \eqref{Eqs1.46AA} simplifies to (note that we again used the fact that $l$ is odd here) 
\Beq\label{Eqs1.46AAA}
\begin{aligned}
   & 2{k\choose l}(A-B)^l+ \sum\limits_{s=1}^{l-1} {k\choose 2l-s}{2l-s\choose l}{l\choose s } (A-B)^{s}(A^{l-s}-(-1)^{s} B^{l-s})  \\
   &-{k\choose l}{2l\choose l}(A^l-B^l)+{2l\choose l}{k\choose 2l}(A^l-B^l) + {2l\choose l}{k+l\choose 2l}(A^l-B^l) +\frac{(2l)!(A^l-B^l)}{(l!)^2}{k\choose l} \\
  &-2(A-B)^l{k\choose l}-\sum\limits_{s=1}^{l-1}{k\choose 2l-s}{2l-s\choose l}{l\choose s} (A-B)^{s}(A^{l-s}-(-1)^{s} B^{l-s})\\ &  + \frac{C(k,k-l) 2^{l}(A^l-B^l)}{l!} - \frac{k!(A^l-B^l)}{(l!)^2(k-2l)!}=0.
\end{aligned}
\Eeq
After canceling out terms, the above equality simplifies to 
\Beq
\begin{aligned}
    \frac{C(k,k-l) 2^{l}(A^l-B^l)}{l!} - \frac{(2l)!(A^l-B^l)}{(l!)^2}{k+l\choose 2l}=0.
\end{aligned}
\Eeq
This implies 
\[
C(k,k-l)=\frac{(k+l)!}{(k-l)! 2^l l!}.
\]
This completes the induction step. A similar argument can be employed for the case of $l$ even and for this reason we skip the proof. This completes the proof of the proposition.
\epr

\bpr[Proof of necessity in Theorem \ref{range}]  Let $f\in C_c^{\infty}(\Bb)$ in $n$ dimensions be a function depending only on the distance from the origin. That is, 
\[
f(x)= \wt{f}(|x|), \mbox{ for some } \wt{f}:[0,\infty)\to \Rb.
\]
We have $\wt{f}\in C^{\infty}([0,\infty))$. 
As before, we do not distinguish between $f$ and $\wt{f}$. The spherical mean transform of $f$ is 
\begin{align*}
    \Rc f(p,t)&=\frac{1}{\o_n}\int\limits_{\Sb^{n-1}} f(|p+t\theta|) \D S(\theta)\\
    &=\frac{1}{\o_n}\int\limits_{\Sb^{n-1}} f(\sqrt{1+t^2+2tp\cdot \theta})\D S(\theta)\\
    &=\frac{\o_{n-1}}{\o_n}\int\limits_{t/2}^{1} f(\sqrt{1+t^2-2ts})(1-s^2)^{\frac{n-3}{2}} \D s.
\end{align*}
The last equality follows from the Funk-Hecke theorem combined with the fact that the support of $f$ is in the unit ball which forces $\frac{t}{2}\leq -p\cdot \theta\leq 1$. Next, employing the change of variable,
\begin{align*}
    1+t^2-2ts = u^2,
\end{align*}
we have 
\begin{align*}
    h(t):= t^{n-2}\Rc f(p,t)&=\frac{\o_{n-1}}{4^{k}\o_n}\int\limits\limits_{|1-t|}^{1} uf(u)\lb 4t^2-(1+t^2-u^2)^2\rb^{k} \D u\\
    &= \frac{\o_{n-1}}{4^{k}\o_n}\int\limits\limits_{|1-t|}^{1} uf(u)\lb 2(u^2+1)t^2-t^4 -(1-u^2)^2\rb^{k}  \D u\\
    &=\frac{\o_{n-1}}{4^{k}\o_n}\int\limits\limits_{|1-t|}^{1} uf(u)\lb Q(t,u)\rb^{k} \D u,
\end{align*}
where we recall that $k=\frac{n-3}{2}$. 
Note that $h$ is infinitely differentiable in $t$. This is clear for $t\neq 1$. However, for $t=1$, we can argue as follows. We have that $h(t)$ (involving the spherical mean transform of a smooth function) is smooth in the $t$ variable and for $t\neq 1$, the derivatives of $h(t)$ can be computed by chain rule. Hence the derivatives of $h(t)$ at $t=1$ can be evaluated by taking the limit as $t\to 1$ of the corresponding derivatives evaluated at $t\neq 1$. The same remark applies for higher order $D$ derivatives instead of ordinary derivatives. With this remark in mind, we will not distinguish between $t=1$ and $t\neq 1$. 

Note that the integral kernel for $h(t)$ is a polynomial in $t$ and $u$. In order to derive a necessary condition for a function $h(t)\in C_c^{\infty}((0,2))$ to be in the range of SMT,  we apply the $D$ operator $k$ times on $h$ and derive a system of equations eliminating integrals with integrand of the form $f(u) u^{m}$ for certain positive integers $m$.  

Also note that up to order $k$, the derivatives are only evaluated on $P(t,u)$, since $Q(t,1-t)=0$. In other words, the integral and derivatives up to order $k$ can be interchanged.

Applying Proposition \ref{NecessityProp}, we have the following necessary condition: 
\begin{align}\label{Eq16.63}
    \Bigg{\{}\sum\limits_{p=0}^{k}C(k,p)  (1-\cdot)^{p} [D^{p} h(\cdot) ] \Bigg{\}}(t) =\Bigg{\{}\sum\limits_{p=0}^{k}C(k,p) (1-\cdot)^{p} [ D^{p}h(\cdot)]\Bigg{\}}(2-t).
\end{align} 
This completes the proof of necessity in Theorem \ref{range}.
\epr

\subsubsection{Proof of sufficiency in Theorem \ref{range}}

In this subsection, we give the proof of sufficiency in Theorem \ref{range}.  We start with the proof of Theorem \ref{Theorem:3.2-Prelim}. We repeat the statement here for reader's convenience. 

  \emph{Let $h\in C_c^{\infty}((0,2))$ satisfy the following evenness condition: 
    \begin{align}\label{EQ3.28}
    \Bigg{\{}\sum\limits_{p=0}^{k} C(k,p) (1-\cdot)^{p} [D^ph](\cdot)\Bigg{\}}(1-t)=\Bigg{\{}\sum\limits_{p=0}^{k} C(k,p) (1-\cdot)^{p} [D^ph](\cdot)\Bigg{\}}(1+t) \mbox{ for all } t\in [0,1].
    \end{align}
    Then $h$ satisfies the following identity: For $\lambda >0$ and $k \geq 0$: 
    \begin{align}
\lb \int\limits_0^{\infty} j_{k+\frac{1}{2}}(\lambda t) th(t)\D t\rb y_{k+\frac{1}{2}}(\lambda)   =\lb \int\limits_0^{\infty} y_{k+\frac{1}{2}}(\lambda t) th(t) \D t\rb j_{k+\frac{1}{2}}(\lambda).
\end{align}}
\bpr[Proof of Theorem \ref{Theorem:3.2-Prelim}]
We define	
\begin{align}
	\notag H(t)&= \sum\limits_{p=0}^{k}C(k,p) (1-t)^{p}[D^{p}h](t).
\end{align}
We observe the following properties of $H(t)$. 
\begin{enumerate}
	\item $H(t)= H(2-t)$ for $0\leq t\leq 1$,
	\item $H(t)= 0$ for $t>2$.
\end{enumerate}
We claim that the integral
\begin{align}
 I_k=\int\limits_{0}^{\infty} H(t) j_{k+\frac{1}{2}} (\lambda(t-1))(t-1) \D t,
\end{align}
vanishes for all $\lambda>0$.

From (2) above, $H(t)$ has non-trivial support only in $(0,2)$. Using this,
\begin{align*}
    I_{k}&=\int\limits_{0}^{2} H(t) j_{k+\frac{1}{2}} (\lambda(t-1))(t-1) \D t\\
    &=\int\limits_{0}^{1} H(t) j_{k+\frac{1}{2}} (\lambda(t-1))(t-1) \D t+\int\limits_{1}^{2} H(t) j_{k+\frac{1}{2}} (\lambda(t-1))(t-1) \D t.
\end{align*}
Substituting $t$ by $2-t$ in the second integral, noting from the formula in Lemma \ref{Lem6.2} that $j_{k+\frac{1}{2}}(x)$ is an even function in $x$, and using (1) above, we have 
\begin{align*}
	0&= I_k=  \int\limits_{0}^{\infty} \sum\limits_{p=0}^{k}C(k,p) (1-t)^{p}[D^{p}h](t) j_{k+\frac{1}{2}} (\lambda(t-1)) (t-1) \D t\\
	&= \int\limits_{0}^{\infty} \sum\limits_{p=0}^{k}C(k,p) th(t) D^p \lb \frac{(t-1)^{p+1}j_{k+\frac{1}{2}}(\lambda(t-1))}{t}\rb  \D t.
\end{align*}
In the last step, we used repeated integration by parts. Next, substituting $t$ by $-t$, we get
\begin{align*}
I_{k}&= -\int\limits_{-\infty}^{0} \sum\limits_{p=0}^{k}(-1)^pC(k,p) D^p \lb \frac{(t+1)^{p+1}j_{k+\frac{1}{2}}(\lambda(t+1))}{t}\rb th(-t) \D t. 
\end{align*}
From Theorem \ref{Sec6:thm} below, we obtain the following equality. This is a technical result and in order not to disturb the flow of proof, we prefer to give it at the end.  
\begin{align*}
I_k=(-1)^{k+1} \int\limits_{-\infty}^{0}\left\{D^k\left(\frac{\sin (\lambda t)}{t}\right) y_{k+\frac{1}{2}}({\lambda})+D^k\left(\frac{\cos(\lambda t)}{t}\right) j_{k+\frac{1}{2}}({\lambda})\right\} th(-t) \D t .
\end{align*}
Again from the formulas in Lemma \ref{Lem6.2}, we see that 
\begin{align*}
& D^{k}\lb \frac{\sin \lambda (\cdot) }{(\cdot)}\rb (-t)=D^{k}\lb \frac{\sin \lambda (\cdot) }{(\cdot)}\rb (t),\\
& D^{k}\lb \frac{\cos \lambda (\cdot) }{(\cdot)}\rb (-t)=-D^{k}\lb \frac{\cos \lambda (\cdot) }{(\cdot)}\rb (t).
\end{align*}
Letting $t\to -t$ in the integral above, we have 
\begin{align*}
    0=(-1)^{k}\int\limits_{0}^{\infty}\left\{D^k\left(\frac{\sin (\lambda t)}{t}\right) y_{k+\frac{1}{2}}({\lambda})-D^k\left(\frac{\cos(\lambda t)}{t}\right) j_{k+\frac{1}{2}}({\lambda})\right\} th(t) \D t.
\end{align*}
Hence 
\begin{align}\label{special-formula}
  \lb\int\limits_0^{\infty} D^k\left(\frac{\sin (\lambda t)}{t}\right) th(t)\D t\rb y_{k+\frac{1}{2}}(\lambda)=\lb \int\limits_0^{\infty} D^{k}\lb \frac{\cos(\lambda t)}{t}\rb \, th((t) \D t\rb j_{k+\frac{1}{2}}(\lambda).
\end{align}
The above formula \eqref{special-formula} can be rewritten in a more symmetric form as follows. For $\lambda>0$ and $k\geq 0$: 
\begin{align}\label{special-formula11}
\lb \int\limits_0^{\infty} j_{k+\frac{1}{2}}(\lambda t) th(t)\D t\rb y_{k+\frac{1}{2}}(\lambda)   =\lb \int\limits_0^{\infty} y_{k+\frac{1}{2}}(\lambda t) th(t) \D t\rb j_{k+\frac{1}{2}}(\lambda).
\end{align}
\epr

\bpr[Proof of sufficiency part of Theorem \ref{range}]
The sufficiency part of proof of the main theorem follows as a straightforward consequence of \eqref{special-formula11} combined with Theorem~\ref{range-AKQ}. Indeed, for $\lambda>0$, the left hand side of \eqref{special-formula11} is the product of the Hankel transform of $g$ (recall that $h(t)=t^{n-2}g(t)$) and the spherical Bessel function of the second kind. Theorem \ref{Theorem:3.2-Prelim} gives that this factors into a product of two functions, one of them being the spherical Bessel function of the first kind in $\lambda$. Since $j_{k+\frac{1}{2}}(\lambda)$ and $y_{k+\frac{1}{2}}(\lambda)$ have no common zeros \cite[eq.(9.5.2)]{Abramowitz-Stegun}, by Theorem~\ref{range-AKQ}, we have the sufficiency part of Theorem \ref{range}.
\epr
\subsubsection{Proof of a technical result used in Theorem \ref{Theorem:3.2-Prelim}}
The following result was used in the proof of Theorem \ref{Theorem:3.2-Prelim} above. 

We recall from \eqref{Bessel1} and \eqref{Bessel2}, the spherical Bessel functions of the first and second kind, respectively, with $\A=k+\frac{1}{2}$ (modulo constants) and 
    $C(k,p)=\frac{(2k-p)!}{p!(k-p)! 2^{k-p}}$. 

\begin{theorem}\label{Sec6:thm}
     For $\lambda>0$ and $t \neq 0$, define 
    \[
    \begin{aligned}
        M_{k}(\lambda)&=\sum\limits_{p=0}^{k} C(k,p) (-1)^{p} D^{p}\lb \frac{(1+t)^{p+1} j_{k+\frac{1}{2}}(\lambda(1+t))}{t}\rb.
        \end{aligned}
        \]
        
        Then $M_k(\lambda)$ has the following decomposition: 
        \Beq \label{Eqs6.1}
        \begin{aligned}
            M_k(\lambda)=
        (-1)^k \left\{D^k\left(\frac{\sin (\lambda t)}{t}\right) y_{k+\frac{1}{2}}({\lambda})+D^k\left(\frac{\cos(\lambda t)}{t}\right) j_{k+\frac{1}{2}}({\lambda})\right\}.
    \end{aligned}
    \Eeq
\end{theorem}

\bpr
We can assume in what follows that $t\neq -1$. The result for the case $t=-1$ will follow from continuity. 

We have, using the expressions from Lemma \ref{Lem6.2},
\begin{align*}
    (-1)^{k}M_{k}&=\sum\limits_{p=0}^{k}\sum\limits_{l=0}^{k}\frac{(-1)^{p} C(k,p) C(k,l)}{\lambda^{2k+1-l}}D^{p}\Bigg{[}\frac{1}{t(1+t)^{2k-l-p} }\\
    &\times \Bigg{\{}\cos \lambda t \Big{\{} (-1)^{l/2}\lb \frac{(-1)^l+1}{2}\rb \sin \lambda +(-1)^{(l+1)/2} \lb \frac{(-1)^{l+1}+1}{2}\rb \cos \lambda \Big{\}}\\
    &+ \sin \lambda t \Big{\{} (-1)^{l/2}\lb \frac{(-1)^l+1}{2}\rb \cos \lambda -(-1)^{(l+1)/2} \lb \frac{(-1)^{l+1}+1}{2}\rb \sin \lambda \Big{\}}\Bigg{\}}\Bigg{]}.
\end{align*}
For simplicity of notation, we will denote the following: 
\begin{align*}
   & U_{l}=\Big{\{} (-1)^{l/2}\lb \frac{(-1)^l+1}{2}\rb \sin \lambda +(-1)^{(l+1)/2} \lb \frac{(-1)^{l+1}+1}{2}\rb \cos \lambda \Big{\}},\\
   &V_{l}=\Big{\{} (-1)^{l/2}\lb \frac{(-1)^l+1}{2}\rb \cos \lambda -(-1)^{(l+1)/2} \lb \frac{(-1)^{l+1}+1}{2}\rb \sin \lambda \Big{\}}.
\end{align*}
Then we have 
\begin{align*}
    (-1)^{k}M_{k}&=\sum\limits_{p=0}^{k}\sum\limits_{l=0}^{k}\frac{(-1)^{p} C(k,p) C(k,l)}{\lambda^{2k+1-l}}D^{p}\Bigg{\{}\frac{1}{t(1+t)^{2k-l-p} }\lb U_l \cos \lambda t  + V_l\sin \lambda t \rb\Bigg{\}}.
\end{align*}
Using the expression for the derivatives from Lemma \ref{Lem6.2}, and after some rearrangements, we get
\Beq
\begin{aligned}           
    (-1)^kM_{k}&=\sum\limits_{p=0}^{k}\sum\limits_{l=0}^{k}\sum\limits_{r=0}^{p}\frac{C(k,p)C(k,l)C(p,r)r!{2k-l-p+r-1\choose r}}{\lambda^{2k+1-l} t^{2p+1-r}(1+t)^{2k-l-p+r}} \Big{\{}U_l\cos \lambda t  + V_l\sin \lambda t \Big{\}} \\ 
   \notag  &-\sum\limits_{p=1}^{k}\sum\limits_{l=0}^{k}\sum\limits_{m=0}^{p-1}\sum\limits_{r=0}^{m}\sum\limits_{s=0}^{p-m-1}\frac{ C(k,p) C(k,l){p\choose m} C(m,r)C(p-m-1,s)r!{2k-p-l+r-1\choose r} }{\lambda^{2k-l-s}t^{2p-r-s}(1+t)^{2k-l-p+r}} \\
   &\times \Bigg{[} U_l\Big{\{}  (-1)^{1+\frac{s}{2}}\lb \frac{(-1)^{s}+1}{2}\rb \sin \lambda t+  (-1)^{1+\frac{s+1}{2}}\lb\frac{(-1)^{s+1}+1}{2}\rb\cos \lambda t\Big{\}}\\
   \notag &-V_l\Big{\{}  (-1)^{1+\frac{s}{2}}\lb \frac{(-1)^{s}+1}{2}\rb \cos \lambda t-  (-1)^{1+\frac{s+1}{2}}\lb\frac{(-1)^{s+1}+1}{2}\rb\sin \lambda t\Big{\}}\Bigg{]}.
\end{aligned}
\Eeq
Note that in the expression above, we have separated the $m=p$ term. Our motivation for doing so is that we want to use the expressions in Lemma \ref{Lem6.2}. When $m<p$, at least one derivative lands on the $\sin$ or $\cos$ term. We carry out one $D$ derivative and then invoke the expressions from Lemma \ref{Lem6.2} for $p-m-1$ derivatives of $\frac{\sin x}{x}$ and $\frac{\cos x}{x}$.
Interchanging the order of summation in the second summand, we get 
\begin{align}
    \notag (-1)^kM_{k}&=\sum\limits_{p=0}^{k}\sum\limits_{l=0}^{k}\sum\limits_{r=0}^{p}\frac{C(k,p)C(k,l)C(p,r)r!{2k-l-p+r-1\choose r}}{\lambda^{2k+1-l} t^{2p+1-r}(1+t)^{2k-l-p+r}} \Big{\{}U_l\cos \lambda t  + V_l\sin \lambda t \Big{\}} \\ 
   \notag  &-\sum\limits_{l=0}^{k}\sum\limits_{s=0}^{k-1}\sum\limits_{p=s+1}^{k}\sum\limits_{r=0}^{p-1-s}\sum\limits_{m=r}^{p-1-s}\frac{ C(k,p) C(k,l){p\choose m} C(m,r)C(p-m-1,s)r!{2k-p-l+r-1\choose r} }{\lambda^{2k-l-s}t^{2p-r-s}(1+t)^{2k-l-p+r}} \\
   \label{Eq45.22}
   &\times \Bigg{[} U_l\Big{\{}  (-1)^{1+\frac{s}{2}}\lb \frac{(-1)^{s}+1}{2}\rb \sin \lambda t+  (-1)^{1+\frac{s+1}{2}}\lb\frac{(-1)^{s+1}+1}{2}\rb\cos \lambda t\Big{\}}\\
   \notag &-V_l\Big{\{}  (-1)^{1+\frac{s}{2}}\lb \frac{(-1)^{s}+1}{2}\rb \cos \lambda t-  (-1)^{1+\frac{s+1}{2}}\lb\frac{(-1)^{s+1}+1}{2}\rb\sin \lambda t\Big{\}}\Bigg{]}.
\end{align}
Next let us simplify the summation in $m$ in the second summand. We use the identity from Lemma \ref{AAI} in \eqref{Eq45.22}.  Then 
\[
\begin{aligned}
\frac{(-1)^{k} M_{k}\lambda^{2k}(t+1)^{2k}}{k!}&=\sum\limits_{p=0}^{k}\sum\limits_{l=0}^{k}\sum\limits_{r=0}^{p}\frac{C(k,l){2k-p\choose k} {2p-r\choose p}{2k-l-p+r-1\choose r}\lambda^{l-1}(t+1)^{l+p-r}}{ 2^{k-r}t^{2p+1-r}}\\
     &\times \Big{\{}U_l\cos \lambda t + V_l\sin \lambda t \Big{\}} \\ 
    &-\sum\limits_{l=0}^{k}\sum\limits_{s=0}^{k-1}\sum\limits_{p=s+1}^{k}\sum\limits_{r=0}^{p-1-s}\frac{ C(k,l){2k-p\choose k}{2p-s-r-1\choose p}{2k-p-l+r-1\choose r}\lambda^{l+s}(t+1)^{l+p-r}}{t^{2p-r-s}(s+1)!2^{k-1-r-s}} \\
 &\times \Bigg{[} U_l\Big{\{}  (-1)^{1+\frac{s}{2}}\lb \frac{(-1)^{s}+1}{2}\rb \sin \lambda t+  (-1)^{1+\frac{s+1}{2}}\lb\frac{(-1)^{s+1}+1}{2}\rb\cos \lambda t\Big{\}}\\
    &-V_l\Big{\{}  (-1)^{1+\frac{s}{2}}\lb \frac{(-1)^{s}+1}{2}\rb \cos \lambda t-  (-1)^{1+\frac{s+1}{2}}\lb\frac{(-1)^{s+1}+1}{2}\rb\sin \lambda t\Big{\}}\Bigg{]}.
    \end{aligned}
\]
    We note that when $s=-1$, the term within square parantheses in the second summand above is precisely $-\lb \cos \lambda t U_l + \sin \lambda t V_l\rb$, and the remaining terms match with the first summand. Therefore the first summand can be absorbed into the second by adding $s=-1$ term in the second. We get, 
    \begin{align*}
        \frac{(-1)^{k} M_{k}\lambda^{2k}(t+1)^{2k}}{k!}&=-\sum\limits_{l=0}^{k}\sum\limits_{s=-1}^{k-1}\sum\limits_{p=s+1}^{k}\sum\limits_{r=0}^{p-1-s}\frac{ C(k,l){2k-p\choose k}{2p-s-r-1\choose p}{2k-p-l+r-1\choose r}\lambda^{l+s}(t+1)^{l+p-r}}{t^{2p-r-s}(s+1)!2^{k-1-r-s}} \\
 &\times \Bigg{[} U_l\Big{\{}  (-1)^{1+\frac{s}{2}}\lb \frac{(-1)^{s}+1}{2}\rb \sin \lambda t+  (-1)^{1+\frac{s+1}{2}}\lb\frac{(-1)^{s+1}+1}{2}\rb\cos \lambda t\Big{\}}\\
    &-V_l\Big{\{}  (-1)^{1+\frac{s}{2}}\lb \frac{(-1)^{s}+1}{2}\rb \cos \lambda t-  (-1)^{1+\frac{s+1}{2}}\lb\frac{(-1)^{s+1}+1}{2}\rb\sin \lambda t\Big{\}}\Bigg{]}.
    \end{align*}
    Re-indexing in $s$ so that it begins from $s=0$, we then get
    \begin{align*}
        \frac{(-1)^{k} M_{k}\lambda^{2k}(t+1)^{2k}}{k!}&=\sum\limits_{l=0}^{k}\sum\limits_{s=0}^{k}\sum\limits_{p=s}^{k}\sum\limits_{r=0}^{p-s}\frac{ C(k,l){2k-p\choose k}{2p-s-r\choose p}{2k-p-l+r-1\choose r}\lambda^{l+s-1}(t+1)^{l+p-r}}{t^{2p-r-s+1}s!2^{k-r-s}} \\
 &\times \Bigg{[} U_l\Big{\{}  -(-1)^{\frac{s+1}{2}}\lb \frac{(-1)^{s+1}+1}{2}\rb \sin \lambda t+  (-1)^{\frac{s}{2}}\lb\frac{(-1)^{s}+1}{2}\rb\cos \lambda t\Big{\}}\\
    &+V_l\Big{\{}  (-1)^{\frac{s+1}{2}}\lb \frac{(-1)^{s+1}+1}{2}\rb \cos \lambda t+ (-1)^{\frac{s}{2}}\lb\frac{(-1)^{s}+1}{2}\rb\sin \lambda t\Big{\}}\Bigg{]}.
    \end{align*}
   For simplicity, we let 
  \Beq\label{Bls-Bu}
    \begin{aligned}
        B_{l,s}&=U_l\Big{\{}  -(-1)^{\frac{s+1}{2}}\lb \frac{(-1)^{s+1}+1}{2}\rb \sin \lambda t+  (-1)^{\frac{s}{2}}\lb\frac{(-1)^{s}+1}{2}\rb\cos \lambda t\Big{\}}\\
    &+V_l\Big{\{}  (-1)^{\frac{s+1}{2}}\lb \frac{(-1)^{s+1}+1}{2}\rb \cos \lambda t+ (-1)^{\frac{s}{2}}\lb\frac{(-1)^{s}+1}{2}\rb\sin \lambda t\Big{\}},
    \end{aligned}
    \Eeq
    where we recall that 
    \begin{align*}
    & U_{l}=\Big{\{} (-1)^{l/2}\lb \frac{(-1)^l+1}{2}\rb \sin \lambda +(-1)^{(l+1)/2} \lb \frac{(-1)^{l+1}+1}{2}\rb \cos \lambda \Big{\}},\\
   &V_{l}=\Big{\{} (-1)^{l/2}\lb \frac{(-1)^l+1}{2}\rb \cos \lambda -(-1)^{(l+1)/2} \lb \frac{(-1)^{l+1}+1}{2}\rb \sin \lambda \Big{\}}.
\end{align*}
We then have 
\begin{align}\label{Eqs3.38}
        \frac{(-1)^{k} M_{k}\lambda^{2k}(t+1)^{2k}}{k!}&=\sum\limits_{l=0}^{k}\sum\limits_{s=0}^{k}\sum\limits_{p=s}^{k}\sum\limits_{r=0}^{p-s}\frac{ C(k,l){2k-p\choose k}{2p-s-r\choose p}{2k-p-l+r-1\choose r}\lambda^{l+s-1}(t+1)^{l+p-r}}{t^{2p-r-s+1}s!2^{k-r-s}} B_{l,s}.
    \end{align}

Replacing $r$ by $p-s-r$ in \eqref{Eqs3.38}, we get,
\begin{align*}
    \frac{(-2)^k \lambda^{2k+1}(1+t)^{2k}M_{k}}{k!}&=\sum\limits_{l=0}^{k}\sum\limits_{s=0}^{k}\sum\limits_{p=s}^{k}\sum\limits_{r=0}^{p-s} \frac{2^{p-r}{2k-p\choose k} C(k,l){p+{r}\choose p}{2k-s-l-1-{r}\choose p-s-{r}}\lambda^{l+s}(1+t)^{l+s+{r}}}{t^{p+{r}+1}s!}B_{l,s}.
   \end{align*}
Let us restrict the sum to those $(l,s)$ such that $l+s=u$, where $0\leq u\leq 2k$. It is straightforward to check that $B_{l,s}$ only depends on $l+s$. If $l+s=u$, sometimes we denote $B_{l,s}$ as $B_u$ for convenience.    We call  this restricted sum on the right above as $S=S(u)$.  If $0\leq u\leq k$, then 
\begin{align}\label{Eq45.83A}
    S=S(u)=\frac{(\lambda(1+t))^{u}k! B_{l,s}}{2^{k-u} u!t}\sum\limits_{s=0}^{u}\sum\limits_{p=s}^{k}\sum\limits_{r=0}^{p-s}\frac{2^{p-r-s}{2k-p\choose k} {2k-u+s\choose k-u+s} {u\choose s}{p+r\choose r}{2k-u-1-r\choose p-s-r}(1+t)^r}{t^{p+r}}. 
    \end{align}
    On the other hand, if $k<u\leq 2k$, we have 
    \begin{align}\label{Eq45.83B}
    S=S(u)=\frac{(\lambda(1+t))^{u}k!B_{l,s}}{2^{k-u} u!t}\sum\limits_{s=u-k}^{k}\sum\limits_{p=s}^{k}\sum\limits_{r=0}^{p-s}\frac{2^{p-r-s}{2k-p\choose k} {2k-u+s\choose k-u+s} {u\choose s} {p+r\choose r}{2k-u-1-r\choose p-s-r}(1+t)^r}{t^{p+r}}. 
    \end{align}
    With this, we have
    \begin{align}\label{Eq6.33}
        M_{k}&=\frac{(-1)^k k!}{2^k\lambda^{2k+1}(1+t)^{2k}}\sum\limits_{u=0}^{2k}S(u).
    \end{align}

    In Lemma \ref{L5.4}, we have a simplified expression for $S(u)$ and using this, we now have 
    \begin{align*}
        M_{k}(\lambda)=\frac{(-1)^{k} (k!)^2}{4^{k}(\lambda t)^{2k+1}}\sum\limits_{u=0}^{2k}\sum\limits_{j=0}^{u} \frac{2^{u}\lambda^u}{u!}{u\choose j}{2k-j\choose k} {2k-u+j\choose k} t^{j} B_{u}.
    \end{align*}
    
    To conclude the proof of Theorem \ref{Sec6:thm}, let us expand the right-hand side of \eqref{Eqs6.1}. We let this  be $\wt{M}_{k}$. That is, 
   \begin{align*}
	\wt{M}_k= (-1)^k \left\{D^k\left(\frac{\sin (\lambda t)}{t}\right) y_{k+\frac{1}{2}}({\lambda})+D^k\left(\frac{\cos(\lambda t)}{t}\right) j_{k+\frac{1}{2}}({\lambda})\right\}.
\end{align*}
Expanding using formulas from Lemma \ref{Lem6.2}, we have 
\[
\begin{aligned}
	 \wt{M}_k &= \frac{(-1)^{k}}{\lambda^{2k+1} t^{2k+1}}\sum_{l=0}^{k} \sum_{m=0}^{k} C(k,l)C(k,m) \lambda^{l+m} t^m \\
	 &\times  \Bigg{\{}\sin\lambda (1+t) (-1)^{\frac{l+m}{2}}\left\{\left(\frac{(-1)^l +1}{2}\right)\left(\frac{(-1)^m+1}{2}\right)+ \left(\frac{(-1)^{l+1} +1}{2}\right)\left(\frac{(-1)^{m+1}+1}{2}\right)\right\}\\
	 &+ \cos\lambda (1+t)(-1)^{\frac{l+m+1}{2}}\left\{ \left(\frac{(-1)^l +1}{2}\right) \left(\frac{(-1)^{m+1}+1}{2}\right)+ \left(\frac{(-1)^{l+1} +1}{2}\right) \left(\frac{(-1)^m+1}{2}\right) \right\}\Bigg{\}}. 
\end{aligned}
\]
Using the expression for $B_{l,s}$ defined earlier, we have 
\begin{align*}
	\wt{M}_k= \frac{(-1)^k}{\lambda^{2k+1} t^{2k+1}}\sum_{l=0}^{k} \sum_{m=0}^{k} C(k,l)C(k,m) \lambda^{l+m} t^m B_{l,m}.
\end{align*}
We now restrict the sum to those $(l,m)$ such that $l+m=u$ with $0\leq u\leq 2k$. Then 
\begin{align*}
    \wt{M}_{k}&=\frac{(-1)^k}{\lambda^{2k+1} t^{2k+1}}\sum\limits_{u=0}^{2k}\lambda^{u} \sum\limits_{m=0}^{u} C(k,u-m)C(k,m)t^{m}B_{u}\\
    &=\frac{(-1)^k (k!)^2}{4^{k} \lambda^{2k+1} t^{2k+1}}\sum\limits_{u=0}^{2k}\sum\limits_{m=0}^{u} \frac{2^u\lambda^u}{u!} \binom{u}{m} \binom{2k-m}{k} \binom{2k-u+m}{k} t^m B_{u}.
\end{align*}
We have shown that $M_k=\wt{M}_k$ and this completes the proof of the theorem. 

\epr 
\subsection{Range characterization for general functions}\label{sec:range-gen}
We now prove the range characterization for a general (not necessarily radial) function by expansion into spherical harmonics. The calculations of the previous proof are going to be crucially used.
\bpr[Proof of Theorem \ref{Thm1.4}]

Following the calculations done in \cite{Salman_Article}, we have the following: 
\begin{align}
   \notag g_{m,l}(t)& = \frac{\o_{n-1}}{4^{\frac{n-3}{2}} t^{n-2} \o_{n} C_{m}^{\frac{n-2}{2}}(1)}\int\limits_{|1-t|}^{1} u f_{m,l}(u)C_{m}^{\frac{n-2}{2}}\lb \frac{1+u^2-t^2}{2u}\rb\Big{\{}\lb (1+t)^2-u^2\rb\lb u^2-(1-t)^2\rb\Big{\}}^{\frac{n-3}{2}} \D u\\
    \label{Eq49.2} &=\frac{\o_{n-1}}{ t^{n-2} \o_{n}C_{m}^{\frac{n-2}{2}}(1)}\int\limits_{|1-t|}^{1} u^{n-2} f_{m,l}(u)C_{m}^{\frac{n-2}{2}}\lb \frac{1+u^2-t^2}{2u}\rb\Bigg{\{}1- \frac{\lb 1+u^2-t^2\rb^2}{4u^2}\Bigg{\}}^{\frac{n-3}{2}} \D u.
    \end{align}
    
    We use the following formula for Gegenbauer polynomials: 
    \[
    C_m^{(\A)}(x)=K (1-x^2)^{-\A+\frac{1}{2}}\frac{\D^{m}}{\D x^{m}}\lb 1-x^{2}\rb^{m+\A-\frac{1}{2}},
    \]
    where 
    \[
    K=\frac{(-1)^{m} \Gamma(\A + \frac{1}{2})\Gamma(m+2\A)}{2^m m!\Gamma(2\A)\Gamma(m+\A+\frac{1}{2})}.
    \]
    By repeated application of chain rule, we have 
    \Beq\label{Eq49.3}
    C_{m}^{\frac{n-2}{2}}\lb \frac{1+u^2-t^2}{2u}\rb=K\lb 1-\lb \frac{1+u^2-t^2}{2u}\rb^2\rb^{-\frac{n-3}{2}}(-u)^{m} D^{m} \lb 1-\frac{\lb 1+u^2-t^2\rb^2}{4u^2}\rb^{m+\frac{n-3}{2}},
    \Eeq
    where, we recall that $D=\frac{1}{t}\frac{\D}{ \D t }$.
    Substituting \eqref{Eq49.3} into \eqref{Eq49.2}, we get, 
    \begin{align*}
        t^{n-2} g_{m,l}(t)=\frac{K(-1)^{m}\o_{n-1}}{\o_n C_{m}^{\frac{n-2}{2}}(1)}\int\limits_{|1-t|}^{1} u^{m+n-2}f_{m,l}(u)D^{m}\lb 1-\frac{\lb 1+u^2-t^2\rb^2}{4u^2}\rb^{m+\frac{n-3}{2}} \D u.
    \end{align*}
    Noting that $k=\frac{n-3}{2}$ and that $D^{m}$ can be taken outside the integral, we get, 
    \begin{align*}
        t^{n-2} g_{m,l}(t)&=\frac{K(-1)^{m}\o_{n-1}}{4^{m+k}\o_n C_{m}^{\frac{n-2}{2}}(1)}D^{m}\int\limits_{|1-t|}^{1} u^{1-m}f_{m,l}(u)\lb 4u^2-\lb 1+u^2-t^2\rb^2\rb^{m+\frac{n-3}{2}} \D u\\
        &=\frac{K(-1)^{m}\o_{n-1}}{4^{m+k}\o_n C_{m}^{\frac{n-2}{2}}(1)}D^{m}\int\limits_{|1-t|}^{1} u^{1-m}f_{m,l}(u)\lb 2(u^2+1)t^2-t^4-(1-u^2)^2\rb^{m+\frac{n-3}{2}} \D u.
    \end{align*}
    We denote 
    \begin{align*}
    & h_{m,l}(t) = t^{n-2}g_{m,l}(t)\\
    &\phi_{m,l}(t)=\int\limits_{|1-t|}^{1} u^{1-m}f_{m,l}(u)\lb 2(u^2+1)t^2-t^4-(1-u^2)^2\rb^{m+\frac{n-3}{2}} \D u.
    \end{align*}
    Then we have 
    \[
    h_{m,l}(t)=\frac{K(-1)^{m}\o_{n-1}}{4^{m+k}\o_n C_{m}^{\frac{n-2}{2}}(1)}D^{m}\phi_{m,l}(t).
    \]
    We make the following observations: 
    \begin{itemize}
        \item $\phi_{m,l}(t)\in C_c^{\infty}((0,2))$,
        \item $\phi_{m,l}(t)$ satisfies 
        \[
        [\Lc_{m+k}\phi_{m,l}](1-t)= [\Lc_{m+k}\phi_{m,l}](1+t),
        \]
        where, we recall that 
        \[
         \Lc_{m+k} = \sum\limits_{p=0}^{m+k} \frac{ (m+k+p)! }{(m+k-p)! p! 2^p} (1-t)^{m+k-p} D^{m+k-p}, \qquad D = \frac{1}{t} \frac{\D}{\D t}.
        \]
    \end{itemize}
The smoothness in the first point follows from the fact that $g_{m,l}(t)$ is a smooth function and $\phi_{m,l}(t)$ is the solution of a linear ODE with smooth coefficients and with zero initial conditions. The fact that the support is strictly in $(0,2)$ is due to the fact that $f_{m,l}\in C^{\infty}([0,1))$ has support strictly away from $1$.  The second point follows from the necessity part of Theorem \ref{range} by replacing $k$ by $m+k$. Hence we have the following necessary condition: There is a function $\phi_{m,l}\in C_c^{\infty}((0,2))$, such that $h_{m,l}(t)= D^{m}\phi_{m,l}(t)$ and $\phi_{m,l}(t)$ satisfies 
    \[
        [\Lc_{m+k}\phi_{m,l}](1-t)= [\Lc_{m+k}\phi_{m,l}](1+t).
        \]
        We note that for each $0\leq l\leq d_m$, $\phi_{m,l}$ satisfies the same ODE.
        
        Next, we show that this condition is also sufficient. Since $\phi_{m,l}(t)\in C_c^{\infty}((0,2))$ and $\phi_{m,l}(t)$ satisfies
        \[
        [\Lc_{m+k}\phi_{m,l}](1-t)= [\Lc_{m+k}\phi_{m,l}](1+t),
        \]
        we have by the sufficiency part of the proof of Theorem \ref{range} that 
        \begin{align*}
\lb \int\limits_0^{\infty} j_{k+m+\frac{1}{2}}(\lambda t) t\phi_{m,l}(t)\D t\rb y_{k+m+\frac{1}{2}}(\lambda)   =\lb \int\limits_0^{\infty} y_{k+m+\frac{1}{2}}(\lambda t) t\phi_{m,l}(t) \D t\rb j_{k+m+\frac{1}{2}}(\lambda).
\end{align*}
Therefore, we have 
\begin{align*}
\lb \int\limits_0^{\infty} D^{m}j_{k+\frac{1}{2}}(\lambda t) t\phi_{m,l}(t)\D t\rb y_{k+m+\frac{1}{2}}(\lambda)   =\lb \int\limits_0^{\infty} D^{m}y_{k+\frac{1}{2}}(\lambda t) t\phi_{m,l}(t) \D t\rb j_{k+m+\frac{1}{2}}(\lambda).
\end{align*}
Integrating by parts, we get 
\begin{align*}
\lb \int\limits_0^{\infty} j_{k+\frac{1}{2}}(\lambda t) t h_{m,l}(t)\D t\rb y_{k+m+\frac{1}{2}}(\lambda)   =\lb \int\limits_0^{\infty} y_{k+\frac{1}{2}}(\lambda t) t h_{m,l}(t) \D t\rb j_{k+m+\frac{1}{2}}(\lambda).
\end{align*}
We have the same expression for each $0\leq l\leq d_m$ and hence the $m^{\mathrm{th}}$ order spherical harmonic term of the Hankel transform of $g$ defined as the orthogonal projection of the Hankel transform of $g$ onto the subspace of spherical harmonics of degree $m$ vanishes at the non-zero zeros of the spherical Bessel function $j_{k+m+\frac{1}{2}}(\lambda)$ satisfying \cite[Condition 4, Theorem 11]{Agranovsky-Finch-Kuchment-range}. We are done with the general case as well.
\epr

\subsection{Counterexample to UCP}\label{proof-CE}
In this subsection we prove Theorem \ref{NUCP-1} and Corollary \ref{NUCP-2}. In both the cases, we consider functions possessing radial symmetry. The proof presented here uses the range characterization (Theorem \ref{range}). In fact, this approach has been employed before, see for instance \cite[Section VI.4]{Natterer_book} where it was used to show that the interior problem of computed tomography is not uniquely solvable. The second proof, see below, 
directly produces the function $f$ claimed in the theorem. Due to the local nature of the operator, the construction of such an $f$ is relatively easy. However, in case of non-local problems, the approach via the range characterization may be better suited.

\begin{proof}[Proof of Theorem \ref{NUCP-1}]
    Let $g \in C_c^\infty((0,2))$ be a non-trivial function such that $h(t) = t^{n-2} g(t)$ satisfies \eqref{RC}. Let $\alpha > 0$ be such that $\alpha < 1-\epsilon$. Let us choose $g$ such that $\mathrm{supp} \,g \subset (\alpha, 1-\epsilon) \cup (1+\epsilon, 2-\alpha)$ (see Lemma \ref{existence} for existence of such a non-trivial function). By Theorem \ref{range}, there exists a unique non-trivial function $f \in C_c^\infty(\Bb)$ possessing radial symmetry, such that $\Rc f (p,t) = g(t)$ and hence $\Rc f(p, t)= 0$ for all $p \in \Sn$ and $ t \in ( 1-\epsilon, 1+\epsilon)$. This $f$ can be represented by the expressions given in Theorem \ref{inversion-FPR}. Since the value of $f$ at a point $x \in \Bb$ depends only on the values of $\Rc f$ on spheres passing through a neighborhood of $x$, we have $f|_U = 0$. The proof is complete. 
\end{proof}

\begin{remark}
    Since $\Rc f (p,t) = 0$ for $t < \alpha$, one can also conclude that $f(x) = 0$ for $|x|>1-\alpha$, using support-type theorems \cite{Ambartsoumian2018}.
\end{remark}

\begin{proof}[Proof of Corollary \ref{NUCP-2}]
    Let $U$ be open set such that $\overline{U} \subset \Bb$, and define $m \coloneqq \inf_{x \in U} |x|$ and $M \coloneqq \sup_{x \in U} |x| < 1$. Invoking Theorem \ref{NUCP-1} with $\epsilon = M$, there exists a non-trivial radial function $f$ such that $f$ vanishes in $\{|x| < M\}$ and $\Rc f$ vanishes for all $t \in (1-M, 1+M)$, i.e., $\Rc f$ vanishes on all spheres intersecting $\{|x| < M\}$. In particular, $f$ vanishes on $U$ and $\Rc f$ vanishes on all spheres intersecting $U$.
\end{proof}

\begin{remark}
In the case of functions possessing radial symmetry, the above counterexample is optimal in the sense that the function necessarily vanishes on all of $\{|x| < M\}$. This can be seen as follows: Due to radial symmetry, if $f$ vanishes in $U$, it vanishes in the annulus $A_U \coloneqq \{x \in \Bb: m < |x| < M\}$. Similarly, if $\Rc f$ vanishes on all spheres intersecting $U$, it vanishes on all spheres passing through $A_U$. In particular, $\Rc f$ vanishes on all spheres passing through $\{|x| < M\}$. The local nature of the inversion formula implies that $f$ vanishes on $\{|x| < M\}$.
\end{remark}

The counterexamples to unique continuation given above rely on the existence of a non-trivial function satisfying the range condition, and having appropriate support. We caution the reader that merely extending a compactly supported function in $(0,1)$ to $(0,2)$ by using the range condition does not ensure that the extended function has the desired support. This poses some technical difficulty. We prove the existence of such a function using basic theory of linear ordinary differential equations with variable coefficients. 

\begin{lemma}\label{existence}
    Let $\epsilon \in (0,1)$ and $\alpha > 0$ such that $\alpha < 1-\epsilon$. There exists a non-trivial function $h \in C_c^\infty((0,2))$ such that $\mathrm{supp} \, h \subset (\alpha, 1-\epsilon) \cup (1+\epsilon, 2-\alpha)$ and satisfies
    \[
    [\Lc_{k} h](1-t) = [\Lc_{k} h](1+t) \quad \mbox{for all} \quad t \in (0,1).
    \]
\end{lemma}

\begin{proof}
    Let us first consider $k=0$. In this case, we want a function supported in $(\alpha, 1-\epsilon) \cup (1+\epsilon, 2-\alpha)$ and satisfies
    \[
    h(1-t) = h(1+t) \quad \mbox{ for all } t \in (0,1).
    \]
    This can be easily done by choosing a smooth function supported in $(1+\epsilon, 2-\alpha)$ and then extending it to $(0,1)$ by the relation given above. This idea also works for $k>0$, with some added technical difficulties. 
    
    Let us now assume $k >0$. The range condition can be written as 
    \begin{align}
        &\sum\limits_{l=0}^k \frac{(-1)^{k-l} (k+l)!}{(k-l)! l! 2^l} t^{k-l} \lb \frac{1}{(1-t)} \frac{\D}{\D t} \rb^{k-l} (h(1-t)) \\
        &\qquad \hspace{2cm}= \sum\limits_{l=0}^k \frac{(-1)^{k-l} (k+l)!}{(k-l)! l! 2^l} t^{k-l} \lb \frac{1}{(1+t)} \frac{\D}{\D t} \rb^{k-l} (h(1+t)). \notag
    \end{align} 
    Let $\Tilde{H} \in C_c^\infty((1,2))$ be such that $\mathrm{supp}(\Tilde{H}) \subset (1+\epsilon, 2-\alpha)$ to be chosen later and for $t \in (0,1)$, denote 
    \[
    G(t) = \sum\limits_{l=0}^k \frac{(-1)^{k-l} (k+l)!}{(k-l)! l! 2^l} t^{k-l} \lb \frac{1}{(1+t)} \frac{\D}{\D t} \rb^{k-l} (\Tilde{H}(1+t)).
    \]
    Then $G \in C_c^\infty((0,1))$ and $\mathrm{supp}(G) \subset (\epsilon, 1-\alpha)$. Let us consider the ODE
    \begin{equation}
        \begin{aligned}
            \begin{cases}
                \sum\limits_{l=0}^k \frac{(-1)^{k-l} (k+l)!}{(k-l)! l! 2^l} t^{k-l} \lb \frac{1}{(1-t)} \frac{\D}{\D t} \rb^{k-l} (H(t)) &= G(t) \quad \mbox{for} \quad t \in (\epsilon, 1- \alpha), \\
                \lb H(\epsilon), H^{(1)}(\epsilon), \dots, H^{(k-1)}(\epsilon)\rb &= 0. 
            \end{cases}
        \end{aligned}
    \end{equation}
    The above ODE can be re-written as 
    \begin{equation}
        \begin{aligned}
            \begin{cases}
                \sum\limits_{l=0}^k a_l(t) \lb \frac{\D}{\D t} \rb^l H(t) &= G(t) \quad \mbox{for} \quad t \in (\epsilon, 1- \alpha), \\
                \lb H(\epsilon), H^{(1)}(\epsilon), \dots, H^{(k-1)}(\epsilon)\rb &= 0,
            \end{cases}
        \end{aligned}
    \end{equation}
    where $a_l$ are rational functions of $t$ smooth in the interval $(\epsilon, 1- \alpha)$. Note that 
    \[
    a_k(t) = \frac{(-1)^k t^k}{(1-t)^k},
    \]
    and thus $\frac{1}{a_k}$ is also smooth in $(\epsilon, 1-\alpha)$. Multiplying throughout by $1/a_k$, the ODE becomes 
    \begin{equation}\label{ODE-simple}
        \begin{aligned}
            \begin{cases}
                H^{(k)}(t) + \sum\limits_{l=0}^{k-1} \frac{a_l(t)}{a_k(t)} \lb \frac{\D}{\D t} \rb^l H(t) &= \lb -1 \rb^k \frac{(1-t)^k}{t^k} G(t) \quad \mbox{for} \quad t \in (\epsilon, 1- \alpha), \\
                \lb H(\epsilon), H^{(1)}(\epsilon), \dots, H^{(k-1)}(\epsilon)\rb &= 0.
            \end{cases}
        \end{aligned}
    \end{equation}
    Next we use the representation for the solution to the above ODE, given in \cite[Ch. 3, eq.(6.2)]{Coddington_Book}. If $\varphi_1, \dots, \varphi_k$ is a basis of solutions to the homogeneous equation
    \[
    H^{(k)}(t) + \sum\limits_{l=0}^{k-1} \frac{a_l(t)}{a_k(t)} \lb \frac{\D}{\D t} \rb^l H(t) = 0,
    \]
    then the solution to \eqref{ODE-simple} is given by
    \begin{equation}
        H(t) = \sum\limits_{j=1}^k \varphi_j(t) \int\limits_{\epsilon}^t \frac{W_j(s)}{W(\varphi_1, \dots, \varphi_k)(s)} \lb -1 \rb^k \frac{(1-s)^k}{s^k} G(s) \, \D s,
    \end{equation}
    where $W(\varphi_1, \dots, \varphi_k)$ is the Wronskian of the basis $\varphi_1, \dots, \varphi_k$ and $W_j(s)$ is obtained from $W(\varphi_1, \dots, \varphi_k)$ by replacing the $j-$th column $\lb \varphi_j, \varphi_j^{(1)}, \dots, \varphi_j^{(k-1)} \rb$ by $(0,0, \dots, 1)$ and then taking the determinant. Due to the support restriction of $G$, $H$ vanishes in a small interval to the right of $t = \epsilon$, and hence all its derivatives vanish at $t =\epsilon$. In particular, $H(\epsilon)= H^{(1)}(\epsilon) = \dots = H^{(k-1)}(\epsilon) = 0$. Thus, by uniqueness, this is the solution of the ODE \eqref{ODE-simple}. 

    We also want the function $H$ and all its derivatives to vanish at $t = 1-\alpha$. To this end, recall that 
    \begin{align}
        G(t) &= \sum\limits_{l=0}^k \frac{(-1)^{k-l} (k+l)!}{(k-l)! l! 2^l} t^{k-l} \lb \frac{1}{(1+t)} \frac{\D}{\D t} \rb^{k-l} (\Tilde{H}(1+t)) \\
        &= \sum\limits_{l=0}^k b_l(t) \lb \frac{\D}{\D t} \rb^l (\Tilde{H}(1+t)).
    \end{align}
    The exact expression of the coefficients $b_l$ is not important, but note that these are rational functions of $t$ smooth in the interval $(\epsilon, 1- \alpha)$. Substituting this into the expression for $H$ and performing integration by parts (no boundary terms due to support condition of $\Tilde{H}$), we obtain
    \begin{align}
        H(1-\alpha) &= \int\limits_{\epsilon}^{1-\alpha} \Phi(s) \Tilde{H}(1+s) \, \D s  
    \end{align}
    for some smooth function $\Phi$.

    If $\Phi \equiv 0$, there is nothing to prove. If not, $\exists s_0 \in (\epsilon, 1-\alpha) $ at which $\Phi(s)$ is either positive or negative and hence by continuity, keeps the same sign in a small interval around $s_0$. Let this interval be $I_0$. Let $I_1, I_2 \subset I_0$ be disjoint. Choose two smooth cut-off functions $\chi_1$ and $\chi_2$ supported in $I_1$ and $I_2$ respectively. For $t \in (1,2)$, let us choose 
    \[
    \Tilde{H}(t) = c_1 \chi_1(t-1)+ c_2 \chi_2 (t-1)
    \]
    for $c_1, c_2$ to be chosen later. We then have
    \begin{align*}
        \int\limits_{\epsilon}^{1-\alpha} \Phi(s) \Tilde{H}(1+s) \, \D s &= c_1 \int\limits_{\epsilon}^{1-\alpha} \Phi(s) \chi_1(s) \, \D s + c_2 \int\limits_{\epsilon}^{1-\alpha} \Phi(s) \chi_2(s) \, \D s \\
        &= c_1 \int\limits_{I_1} \Phi(s) \chi_1(s) \, \D s + c_2 \int\limits_{I_2} \Phi(s) \chi_2(s) \, \D s. 
    \end{align*}
    Choosing $c_1 = - \int\limits_{I_2} \Phi(s) \chi_2(s) \, \D s$ and $c_2 = \int\limits_{I_1} \Phi(s) \chi_1(s) \, \D s $, we get
    \begin{align*}
        H(1-\alpha) &= \int\limits_{\epsilon}^{1-\alpha} \Phi(s) \Tilde{H}(1+s) \, \D s \\
        &= 0.
    \end{align*}
    In fact, due to the choice of support of $\Tilde{H}$, $H$ vanishes in a small interval to the left of $t = 1-\alpha$ and hence all its derivatives also vanish at $t=1-\alpha$. Thus, the function $H$, defined in $(\epsilon, 1- \alpha)$, obtained above can be extended by $0$ to a smooth function in $(0,1)$. Finally, the function $h \in C_c^\infty((0,2))$ defined as
    \begin{equation}
        h(t) = \begin{cases}
            H(1-t), \quad \mbox{for} \quad t \in (0,1), \\
            \Tilde{H}(t), \quad \mbox{for} \quad t \in (1,2),
        \end{cases}
    \end{equation}
    satisfies the assumptions of the lemma.
\end{proof}
Finally, we present an alternate proof of Theorem \ref{NUCP-1}. 
\begin{proof}[Proof of Theorem \ref{NUCP-1}]
    Recall that when $f$ has radial symmetry, we have \eqref{radon-radial}:
    \begin{align*}
    \Rc f(p,t) &= \frac{\o_{n-1}}{\o_n} \int\limits_{-1}^1 f \lb \sqrt{1+t^2 +  2st} \rb (1-s^2)^k \, \D s.
\end{align*}
Consider the change of variables $u = \sqrt{1+t^2 +  2st}$ to get
\begin{align*}
    \Rc f(p,t) &= \frac{\o_{n-1}}{\o_n} \frac{1}{t} \int\limits_{|1-t|}^{1+t} u f(u) \lb 1 - \lb \frac{u^2-1-t^2}{2t}\rb^2\rb^k \, \D u.
\end{align*}

\noindent Choose $F \in C_c^\infty ((0,1))$ such that $\mathrm{supp}(F) \subset (\epsilon, 1)$ and take $f(t) = \frac{\D^{m}}{\D t^{m}}F(t)$ for any $m \geq 4k+2$. With this choice of $f$, we have for $t \in (1-\epsilon, 1+\epsilon)$
\begin{align*}
    \Rc f(p,t) &= \frac{\o_{n-1}}{\o_n} \frac{1}{t} \int\limits_{\epsilon}^{1} u \lb \frac{\D^{m}}{\D u^{m}}F(u)\rb \lb 1 - \lb \frac{u^2-1-t^2}{2t}\rb^2\rb^k \, \D u
\end{align*}
due to the choice of support of $F$. Performing repeated integration by parts, we obtain that $\Rc f(p,t) = 0$ for all $p \in \Sn$ and $t \in (1-\epsilon, 1+\epsilon)$. 
\end{proof}

\section{Concluding remarks and further directions}\label{further}
\begin{itemize}
\item In this article, we have given a complete range characterization for the SMT in odd dimensions. A complementary simple range description of the SMT in even dimensions has been obtained in our work \cite{AAKS-even}. As opposed to the range conditions in odd dimensions, which use differential operators, the range conditions in even dimensions are based on symmetry relations that utilize integral operators. This discrepancy is in some sense natural, given the distinct properties of the SMT in spaces with dimensions of different parity. For example, it is well known that SMT in odd dimensions has a local inversion, while in even dimension the inversion of that transform is non-local.

\item One of the results of this paper is a counterexample to UCP for SMT in odd dimensions. The authors believe that the UCP (as introduced in this article) should hold in even dimensions, while the interior problem  (see \cite{Natterer_book}) should not have a unique solution there. The authors plan to address these questions in a future work.

\item An offshoot of the current work is the discovery of explicit inversion formulas for the SMT that we study, similar in spirit to the works of Norton \cite{norton1980reconstruction}, Norton-Linzer \cite{norton1981ultrasonic}, Xu-Wang \cite{xu2002time} and others based on Fourier series/spherical harmonics and Hankel transform. Our inversion formulas are valid in all odd and even dimensions, and are simpler than some of the already existing ones. We plan to report this work in an upcoming article.
\end{itemize}

\section*{Acknowledgements}
GA was partially supported by the NIH grant U01-EB029826.

DA was supported by the Research Council of Finland (Flagship of Advanced Mathematics for Sensing Imaging and Modelling grant 359208 and other grant 360434). 

VK would like to thank the Isaac Newton Institute for Mathematical Sciences, Cambridge, UK, for support and hospitality during the workshop, \emph{Rich and Nonlinear Tomography - a multidisciplinary approach} in 2023 where part of this work was done (supported by EPSRC Grant Number EP/R014604/1). 

All the authors thank Mark Agranovsky, Peter Kuchment, Leonid Kunyansky, Todd Quinto, Rakesh and Boris Rubin for several fruitful discussions while this work was being done.

\appendix
\section{Some combinatorial identities}

\begin{lemma}\label{CombinatorialLemmas}
    We have the following identities: 
    \begin{enumerate} 
   \item[(a)] For any $k,l,s\geq 0$ with $l-s\geq 0$,  \begin{align}\label{Eqs1.57AA}
\sum\limits_{m=0}^{l-s} (-1)^m {k+m\choose 2l-s}{l-s \choose m}&= (-1)^{l-s}{k\choose l}. 
\end{align}

\item[(b)] For any $A$ and $B$ and for any $l\geq 0$,
\[
\sum\limits_{s=0}^{l} (-1)^{s}{2l-s \choose l}{l\choose s} (A-B)^{s}(A^{l-s}-(-1)^{s} B^{l-s})=0.
\]
\end{enumerate} 
\end{lemma}
\begin{proof}

We prove (a). Using Vandermonde identity \cite{Concrete_Mathematics}: 
\[
{k+m \choose 2l-s}=\sum\limits_{j=0}^{m} {k\choose 2l-s-j}{m\choose j}. 
\] 
Then 
\begin{align*}
\sum\limits_{m=0}^{l-s} (-1)^m {k+m\choose 2l-s}{l-s \choose m}& \NT\NT=\NT\NT\sum\limits_{m=0}^{l-s} (-1)^m\sum\limits_{j=0}^{m} {k\choose 2l-s-j}{m\choose j}{l-s \choose m}\\
&\NT\NT=\NT\NT\sum\limits_{m=0}^{l-s} (-1)^m\sum\limits_{j=0}^{m} {k\choose 2l-s-j}{l-s \choose j}{l-s-j\choose l-s-m}. 
\end{align*}
In the last equality, we have used the standard fact: 
\[
{a\choose b}{b\choose c}={a\choose c}{a-c\choose b-c}={a\choose c}{a-c\choose a-b}.
\]
Let us interchange the order of summation. We then get,
\begin{align}
  \label{Eqs1.51}\sum\limits_{m=0}^{l-s} (-1)^m {k+m\choose 2l-s}{l-s \choose m}&= \sum\limits_{j=0}^{l-s}\sum\limits_{m=j}^{l-s}(-1)^m {k\choose 2l-s-j}{l-s \choose j}{l-s-j\choose l-s-m}.
\end{align}

Next, we write
\begin{align*}
    & \sum\limits_{j=0}^{l-s}\sum\limits_{m=j}^{l-s}(-1)^m {k\choose 2l-s-j}{l-s \choose j}{l-s-j\choose l-s-m}\\
    &=\sum\limits_{j=0}^{l-s-1}\sum\limits_{m=j}^{l-s}(-1)^m {k\choose 2l-s-j}{l-s \choose j}{l-s-j\choose l-s-m}+ (-1)^{l-s} {k\choose l}\\
    &=\sum\limits_{j=0}^{l-s-1}{k\choose 2l-s-j}{l-s \choose j}\sum\limits_{m=j}^{l-s}(-1)^m {l-s-j\choose l-s-m}+ (-1)^{l-s} {k\choose l}.
\end{align*}
We have that, as long as $j<l-s$, 
\[
\sum\limits_{m=j}^{l-s}(-1)^m {l-s-j\choose l-s-m}=0.
\]
Hence 
\begin{align*}
    \sum\limits_{m=0}^{l-s} (-1)^m {k+m\choose 2l-s}{l-s \choose m}=(-1)^{l-s} {k\choose l}.
\end{align*}
This completes the proof of (a). 

Next, we prove (b). We first split the left hand side as follows: 
\begin{align*}
\notag \sum\limits_{s=0}^{l} (-1)^{s}{2l-s \choose l}{l\choose s} (A-B)^{s}(A^{l-s}-(-1)^{s} B^{l-s})
&=\sum\limits_{s=0}^{l} {2l-s \choose l}{l\choose s} (B-A)^{s}A^{l-s}\\
&-\sum\limits_{s=0}^{l} {2l-s \choose l}{l\choose s} (A-B)^{s}B^{l-s}.
\end{align*}
Here and in several instances throughout the rest of the paper, we use contour integration technique to evaluate combinatorial sums pioneered by Egorychev \cite{Egorychev}.  From Section \ref{Egorychev}, we can write for $\ve>0$, 
\[
{2l-s \choose l} = \frac{1}{2\pi \I}\int\limits_{|z|=\ve} \frac{(1+z)^{2l-s}}{z^{l+1}}\D z,
\]
and hence 
\begin{align*}
\notag \sum\limits_{s=0}^{l} {2l-s \choose l}{l\choose s} (B-A)^{s}A^{l-s}&=\frac{1}{2\pi \I}\int\limits_{|z|=\ve}\sum\limits_{s=0}^{l} {l\choose s} (B-A)^s A^{l-s} \frac{(1+z)^{2l-s}}{z^{l+1}}\D z\\
\notag &=\frac{1}{2\pi \I}\int\limits_{|z|=\ve}\sum\limits_{s=0}^{l} {l\choose s} \lb\frac{B-A}{1+z}\rb ^s A^{l-s} \frac{(1+z)^{2l}}{z^{l+1}}\D z\\
\notag &=\frac{1}{2\pi \I}\int\limits_{|z|=\ve} \lb A+ \frac{B-A}{1+z}\rb^{l} \frac{(1+z)^{2l}}{z^{l+1}}\D z\\
&=\frac{1}{2\pi \I}\int\limits_{|z|=\ve} \frac{((B+Az)(1+z))^{l}}{z^{l+1}} \D z.
\end{align*}
Expanding $(B+Az)^{l} (1+z)^l$ using binomial theorem, we get,
\begin{align*}
    (B+Az)^{l} (1+z)^l=\sum\limits_{u,v=0}^{l} {l \choose u}{l \choose v}B^{u}A^{l-u}z^{l-u} z^{v}.
\end{align*}
Then 
\begin{align*}
    \frac{(B+Az)^{l} (1+z)^l}{z^{l+1}}=\frac{\sum\limits_{u,v=0}^{l} {l \choose u}{l \choose v}B^{u}A^{l-u}z^{v-u}}{z}.
\end{align*}
Hence 
\begin{align*}
    \frac{1}{2\pi \I}\int\limits_{|z|=\ve} \frac{((B+Az)(1+z))^{l}}{z^{l+1}} \D z = \sum\limits_{u=0}^{l} {l \choose u}^2 B^{u} A^{l-u},
\end{align*}
by Cauchy's theorem combined with the fact that for any negative power of $z$ that is not $-1$, the integral vanishes, since $z^{-p}$ has a primitive in a neighborhood of $|z|=\ve$ for $p\neq 1$. Similarly, 
\begin{align*}
    \sum\limits_{s=0}^{l} {2l-s \choose l}{l\choose s} (A-B)^{s}B^{l-s}=\frac{1}{2\pi \I}\int\limits_{|z|=\ve} \frac{((A+Bz)(1+z))^{l}}{z^{l+1}} \D z.
\end{align*}
Exactly the same argument gives
\begin{align*}
    \frac{1}{2\pi \I}\int\limits_{|z|=\ve} \frac{(A+Bz)^{l} (1+z)^l}{z^{l+1}}\D z= \sum\limits_{u=0}^{l} {l \choose u}^2A^{u} B^{l-u}.
\end{align*}
Since ${l \choose u} = {l \choose l-u}$, these two sums are the same. This concludes the proof of (b). 
\end{proof}

\begin{lemma}\label{Lem6.2} We have 
\begin{align}
    & D^p\left( \frac{\sin x}{x} \right)= \sum_{l=0}^{p} \frac{C(p,l)x^l}{x^{2p+1}} \left\{ \sin x \left(\frac{(-1)^l+1}{2}\right)(-1)^{p+\frac{l}{2}}+ \cos x\left(\frac{(-1)^{l+1}+1}{2}\right)(-1)^{p+\frac{l+1}{2}} \right\}\\
    &D^p\left( \frac{\cos x}{x} \right)= \sum_{l=0}^{p} \frac{C(p,l)x^l}{x^{2p+1}} \left\{ \cos x \left(\frac{(-1)^l+1}{2}\right)(-1)^{p+\frac{l}{2}}- \sin x\left(\frac{(-1)^{l+1}+1}{2}\right)(-1)^{p+\frac{l+1}{2}} \right\}\\
     & \mbox{For } d\geq 0, D^m \lb \frac{1}{t(t+1)^d} \rb = (-1)^m  \sum\limits_{r=0}^{m} \frac{C(m,r) \binom{d+r-1}{r} r!}{t^{2m+1-r} (t+1)^{d+r}},
\mbox{ with the convention that } {-1\choose 0}=1. 
\end{align}
\end{lemma}
The proofs of these formulas follow in a straightforward manner by induction and will be skipped.
\begin{lemma}\label{AAI}Denote by 
\Beq\label{Eqs2.28} 
C:=\sum\limits_{m=r}^{p-1-s}\frac{1}{p-m}{2m -r \choose m-r}{2(p-1-m)-s\choose p-1-m-s}.
\Eeq
Then 
\[
C=\frac{1}{s+1}{2p-r-s-1\choose p}. 
\]
\end{lemma}
\bpr
This follows directly from the Abel-Aigner identity. For the sake of completeness, we give the proof. The Abel-Aigner identity (see \cite{Aigner,Concrete_Mathematics}) is as follows: 
\begin{align}\label{Eqs2.29}
    \sum\limits_{k} \frac{r}{tk+r}{tk+r \choose k}{t(n-k)+s\choose n-k} ={tn+r+s\choose n}.
\end{align}
We have 
\begin{align*}
    C&=\sum\limits_{m=r}^{p-1-s} \frac{1}{p-m}{2(m-r)+r\choose m-r} {2(p-1-s-m)+s\choose p-1-s-m}\\
    &=\sum\limits_{m=0}^{p-1-s-r}\frac{1}{p-m-r}{2m+r\choose m}{2(p-1-s-r-m)+s\choose p-1-s-r-m}\\
    &=\sum\limits_{m=0}^{p-1-s-r}\frac{1}{s+1+m}{2(p-1-s-r-m)+r\choose p-1-s-r-m}{2m+s\choose m}\\
    &=\sum\limits_{m=0}^{p-1-s-r}\frac{1}{2m+s+1}{2(p-1-s-r-m)+r\choose p-1-s-r-m}{2m+s+1\choose m}.
\end{align*}
In the last but one step, we have replaced the index $m$ by $p-1-s-r-m$ and in the last step, we have used the following equality, 
\begin{align*}
    \frac{1}{m+s+1}{2m+s\choose m}=\frac{1}{2m+s+1}{2m+s+1\choose m}. 
\end{align*}
Now using Abel-Aigner identity \eqref{Eqs2.29}, we get, 
\begin{align*}
    C= \frac{1}{s+1}{2(p-1-s-r)+r+s+1\choose p-1-s-r}=\frac{1}{s+1}{2p-s-r-1\choose p-1-s-r}=\frac{1}{s+1}{2p-s-r-1\choose p}.
\end{align*}
This completes the proof of Lemma \ref{AAI}.
\epr

  \begin{lemma}\label{L5.4} The expressions $S(u)$ for $0\leq u\leq 2k$ in \eqref{Eq45.83A} and \eqref{Eq45.83B} simplify to 
    \Beq \label{S(u) simplified}
    S(u)=\frac{k!(1+t)^{2k}}{2^k t^{2k+1}}\sum\limits_{j=0}^{u} \frac{2^{u}\lambda^u}{u!}{u\choose j}{2k-j\choose k} {2k-u+j\choose k} t^{j} B_{u},
    \Eeq
    where $B_u=B_{l,s}$ is given in \eqref{Bls-Bu}.
    \end{lemma}
    \bpr 
Replacing the index $s$ by $u-s$ in \eqref{Eq45.83A}, we get,  
\begin{align}\label{Eq45.83CC}
    S&=\frac{(\lambda(1+t))^{u}k!B_{l,s}}{2^{k}u!t}\sum\limits_{s=0}^{u}\sum\limits_{p=u-s}^{k}\sum\limits_{r=0}^{p-u+s}\frac{2^{p-r+s}{2k-p\choose k-p}{2k-s\choose k-s}{u\choose s} {p+r\choose r}{2k-u-1-r\choose p-u+s-r}(1+t)^r}{t^{p+r}}.
    \end{align}
    Similarly, we replace the index $s$ by $u-s$ in \eqref{Eq45.83B}. We then get, 
    \begin{align}\label{Eq45.83DD}
    S=\frac{(\lambda(1+t))^{u}k!B_{l,s}}{2^{k} u!t}\sum\limits_{s=u-k}^{k}\sum\limits_{p=u-s}^{k}\sum\limits_{r=0}^{p-u+s}\frac{2^{p-r+s}{2k-p\choose k} {2k-s\choose k-s} {u\choose s} {p+r\choose r}{2k-u-1-r\choose p-u+s-r}(1+t)^r}{t^{p+r}}. 
    \end{align}
   We aim to simplify \eqref{Eq45.83CC} and \eqref{Eq45.83DD} further. This will be achieved in a few steps below. Before we begin, let us make a remark which is important in the computations below. All the combinatorial terms  in \eqref{Eq45.83CC} and \eqref{Eq45.83DD} have non-negative entries except ${2k-u-r-1\choose p-u+s-r}$ and this term attains 
    ${-1\choose 0}$ which according to our convention is $1$. In the calculations below, it is more convenient to work with ${2k-u-1-r\choose 2k-1-p-s}$. For non-negative entries this is obviously the same as ${2k-u-r-1\choose p-u+s-r}$, but since we have to deal with the case ${-1\choose 0}$ as well, we need to interpret this combinatorial term appropriately. Using the relation ${n\choose k}={n\choose n-k}$, we rewrite this as ${-1\choose -1}$.  Using Egorychev's contour integral approach,
    we can interpret ${-1\choose -1}=1$ based on the following: 
    \Beq\label{Contourintegral2}
    {n\choose k}=\frac{1}{2\pi \I}\int\limits_{|z|=\ve} \frac{1}{(1-z)^{k+1} z^{n-k+1}} \D z.
    \Eeq
    For $n=k=-1$, this integral is $1$. In the calculations below, we always interpret the combinatorial term ${2k-u-1-r\choose 2k-1-p-s}$ based on the above contour integral \eqref{Contourintegral2}. 
    We rewrite \eqref{Eq45.83CC} as,  
\begin{align}\label{Eq45.83C}
    S&=\frac{(\lambda(1+t))^{u}k!B_{l,s}}{2^{k}u!t}\sum\limits_{s=0}^{u}\sum\limits_{p=u-s}^{k}\sum\limits_{r=0}^{p-u+s}\frac{2^{p-r+s}{2k-p\choose k-p}{2k-s\choose k-s}{u\choose s} {p+r\choose r}{2k-u-1-r\choose 2k-1-p-s}(1+t)^r}{t^{p+r}}.
    \end{align}
    Similarly, we rewrite \eqref{Eq45.83DD} as 
    \begin{align}\label{Eq45.83D}
    S=\frac{(\lambda(1+t))^{u}k!B_{l,s}}{2^{k} u!t}\sum\limits_{s=u-k}^{k}\sum\limits_{p=u-s}^{k}\sum\limits_{r=0}^{p-u+s}\frac{2^{p-r+s}{2k-p\choose k} {2k-s\choose k-s} {u\choose s} {p+r\choose r}{2k-u-1-r\choose 2k-1-p-s}(1+t)^r}{t^{p+r}}. 
    \end{align}
    
     Next, let us focus our attention on 
\begin{align}
    S_1:=\sum\limits_{p=u-s}^{k}\sum\limits_{r=0}^{p-u+s}\frac{2^{p-r}{2k-p\choose k} {p+r\choose r}{2k-u-1-r\choose 2k-1-p-s}(1+t)^r}{t^{p+r}}.
\end{align}
We write $S_1$ as follows: 
\Beq\label{Eqs6.43}
\begin{aligned}
    S_1=\sum\limits_{p=u-s}^{k}\sum\limits_{r=0}^{p-u+s}\frac{1}{(2\pi\I)^3}\int\limits_{|z|=\ve_1}\int\limits_{|w|=\ve_2}\int\limits_{|v|=\ve_3} & 2^{p-r}\frac{1}{(1-z)^{k+1} z^{k-p+1}}\frac{1}{(1-w)^{p+1}w^{r+1}}\\
    & \times \frac{1}{(1-v)^{2k-p-s} v^{-u+s-r+p+1}}\frac{(1+t)^{r}}{t^{p+r}}\D z \D w \D v,
\end{aligned}
\Eeq
for suitably chosen $\ve_1,\ve_2,\ve_3$ (given below). 

Note that the right-hand side of \eqref{Eqs6.43} vanishes when $r>p-u+s$,  $p>k$ or $p<u-s$. For, when $r>p-u+s$, the integral in $v$ is $0$ by Cauchy's theorem. When $p>k$, the integral in $z$ is $0$ for the same reason, and for $p<u-s$, the integral in $v$ is $0$. Hence in computing the integral in \eqref{Eqs6.43}, we can let the upper limits of $r,p$ to be $\infty$ and the lower limit of $p$ to be $0$. Later on, we will sum in the $s$ variable as well. Note that due to the presence of the combinatorial term ${2k-s\choose k-s}$, we can let the upper limit of $s$ to be $u$ regardless of whether $0\leq u\leq k$ or $k<u\leq 2k$. Furthermore, in the case when $k<u\leq 2k$, see \eqref{Eq45.83D}, we can let the lower limit of $s$ to be $0$ as well, since in \eqref{Eqs6.43}, the integral in $v$ is $0$.

Hence we can unify \eqref{Eq45.83C} and \eqref{Eq45.83D} into a single expression for all $0\leq u\leq 2k$: 
\Beq
S=\frac{(\lambda(1+t))^{u}k!B_{l,s}}{2^{k} u!t}\sum\limits_{s=0}^{u}\sum\limits_{p=0}^{\infty}\sum\limits_{r=0}^{\infty}\frac{2^{p-r+s}{2k-p\choose k} {2k-s\choose k-s} {u\choose s} {p+r\choose r}{2k-u-1-r\choose 2k-1-p-s}(1+t)^r}{t^{p+r}},
\Eeq
and $S_1$ is now rewritten as  
\[
S_1=\sum\limits_{p=0}^{\infty}\sum\limits_{r=0}^{\infty}\frac{2^{p-r}{2k-p\choose k} {p+r\choose r}{2k-u-1-r\choose 2k-1-p-s}(1+t)^r}{t^{p+r}},
\]
interpreted in terms of the contour integral: 

\Beq\label{Eqs6.43A}
\begin{aligned}
    S_1=\sum\limits_{p=0}^{\infty}\sum\limits_{r=0}^{\infty}\frac{1}{(2\pi\I)^3}\int\limits_{|z|=\ve_1}\int\limits_{|w|=\ve_2}\int\limits_{|v|=\ve_3} & 2^{p-r}\frac{1}{(1-z)^{k+1} z^{k-p+1}}\frac{1}{(1-w)^{p+1}w^{r+1}}\\
    & \times \frac{1}{(1-v)^{2k-p-s} v^{-u+s-r+p+1}}\frac{(1+t)^{r}}{t^{p+r}}\D z \D w \D v.
\end{aligned}
\Eeq

We now establish the choice of contours in \eqref{Eqs6.43A}. The contours will be determined based on taking $t$ fixed. Recall that we have $t\neq 0$ in the statement of the theorem. We will assume that $t\neq -1$ as well. Equation \eqref{Eqs6.44} below is obtained by performing summation in $p$ and $r$ variable. In order for the series to converge, we choose contours such that
\[
|v|<\left| \frac{2tw}{1+t}\right| \mbox{ and } \left| \frac{2z(1-v)}{(1-w)vt}\right|<1.
\]
With $t$ arbitrary, but fixed, choose $|w|=\ve_2\ll 1$ and $|v|=\ve_3\ll 1$ both positive so that $\ve_3< \frac{2|t|\ve_2}{|1+t|}$. 
Next, choose $|z|=\ve_1\ll 1$ so that $\frac{2\ve_1(1+\ve_3)}{(1-\ve_2) \ve_3 |t|}<1$. Then 
\[
\left| \frac{2z(1-v)}{(1-w)vt}\right|\leq \frac{2\ve_1(1+\ve_3)}{(1-\ve_2) \ve_3 |t|}<1.
\]
We have 
\begin{align}\label{Eqs6.44}
    S_1=\frac{2t^2}{(2\pi\I)^3}\iiint \frac{1}{(1-z)^{k+1} z^{k+1}}\frac{v^{u-s}}{(1-v)^{2k-s}}\frac{1}{t(1-w)v-2z(1-v)}\frac{1}{2tw-v(1+t)} \D z \D w \D v.
\end{align}
By choosing $|z| =  \ve_1$ small enough, we can make $w= 1- \frac{2z(1-v)}{tv}$ an external pole. Therefore, performing integration in $w$ using residue theorem, we get 
\begin{align*}
    S_1&=\frac{2t}{(2\pi\I)^2}\iint \frac{1}{(1-z)^{k+1} z^{k+1}}\frac{v^{u-s}}{(1-v)^{2k-s}}\frac{1}{2tv-v^2(1+t)-4z(1-v)} \D z \D v\\
    &=-\frac{2t}{(2\pi\I)^2(t+1)}\iint \frac{1}{(1-z)^{k+1} z^{k+1}}\frac{v^{u-s}}{(1-v)^{2k-s}}\frac{1}{v^2-\frac{2tv}{t+1} +\frac{4z}{1+t}(1-v)} \D z \D v\\
    &=-\frac{2t}{(2\pi\I)^2(t+1)}\iint \frac{1}{(1-z)^{k+1} z^{k+1}}\frac{v^{u-s}}{(1-v)^{2k-s}}\frac{1}{\lb v -\frac{t+2z}{t+1}-\frac{\sqrt{t^2+4z^2-4z}}{t+1}\rb}\\
    &\times \frac{1}{\lb v -\frac{t+2z}{t+1}+\frac{\sqrt{t^2+4z^2-4z}}{t+1}\rb} \D z \D v.
\end{align*}
We have that 
\Beq\label{Eqs6.39A}
v=\frac{t+2z}{t+1}-\frac{\sqrt{t^2+4z^2-4z}}{t+1},
\Eeq
is a simple pole. Reducing $\ve_1$ if necessary, we can ensure that this pole is in the interior of $|v|=\ve_3$, since $v$ in \eqref{Eqs6.39A} can be written in the form,
\[
v=\frac{t+2z}{t+1}-\frac{\sqrt{(t+2z)^2-4z(t+1)}}{t+1}.
\]
The other root of $v$ can be made an external pole by choosing $\ve_1$ small enough.
Integrating in $v$, we get, 
\begin{align*}
    S_1=\frac{t (t+1)^{2k-u}}{2\pi\I}\int \frac{1}{(1-z)^{k+1} z^{k+1}}\frac{\lb (t+2z)-\sqrt{t^2+4z^2-4z}\rb^{u-s}}{(1-2z+\sqrt{t^2+4z^2-4z})^{2k-s}}\frac{1}{\sqrt{t^2+4z^2-4z}} \D z.
\end{align*}
As in \cite{Riedel_Notes}, we  make the change of variable $z(1-z)=\eta$, and we have that the image of $|z|=\ve_1$ is a closed contour which makes one complete turn with origin in its interior and which can be deformed to a circle. We have
\[
z=\frac{1-\sqrt{1-4\eta}}{2}.
\]
Then 
\begin{align*}
    S_1=\frac{t (t+1)^{2k-u}}{2\pi\I}\int \frac{1}{\eta^{k+1}}\frac{\lb t+1-\sqrt{1-4\eta}-\sqrt{t^2-4\eta}\rb^{u-s}}{(\sqrt{1-4\eta}+\sqrt{t^2-4\eta})^{2k-s}}\frac{1}{\sqrt{t^2-4\eta}\sqrt{1-4\eta}} \D \eta.
\end{align*}
For simplicity of notation, we let 
\[
\A = \sqrt{1-4\eta},\quad \B=\sqrt{t^2-4\eta}.
\]
Next let us perform summation in $s$ variable. Recall from the earlier discussion that we can let the lower and upper limits of $s$ to be $0$ and $u$, respectively, regardless of whether $0\leq u\leq k$ or $k<u\leq 2k$.  We get, using the integral representation from Subsection \ref{Egorychev},
\begin{align*}
    S_2&:=\sum\limits_{s=0}^{u} 2^{s}{u\choose s} {2k-s \choose k}S_1\\
    &=\frac{t (t+1)^{2k-u}}{(2\pi\I)^2}\iint \frac{1}{(1-w)^{k+1} w^{k+1}\eta^{k+1}}\frac{{\lb t+1-\A -\B +2w(\A + \B)\rb}^{u}}{(\A + \B)^{2k}}\frac{1}{\A \B} \D \eta \D w.
\end{align*}
As before, let us make the change of variable $w(1-w)=\g$. Then we have 
\begin{align*}
    S_2&=\frac{t (t+1)^{2k-u}}{(2\pi\I)^2}\iint \frac{1}{(\g \eta)^{k+1}} \frac{\lb (t+1-(\sqrt{1-4\g}) (\A + \B)\rb^{u}}{(\A + \B)^{2k}}\frac{1}{\A \B} \frac{1}{\sqrt{1-4\g}} \D \eta \D \g\\
    &=\frac{t (t+1)^{2k-u}}{(2\pi\I)^2}\sum\limits_{s=0}^{u}(-1)^{u+s} {u\choose s} \iint \frac{1}{(\g \eta)^{k+1}} \frac{(t+1)^s}{(\A + \B)^{2k-u+s}}\frac{1}{\A \B} \frac{1}{\lb \sqrt{1-4\g}\rb^{1+s-u}} \D \eta \D \g\\
    &=\frac{t (t+1)^{2k-u}}{(2\pi\I)^2}\sum\limits_{s=0}^{u}(-1)^{u+s} {u\choose s} (t+1)^s\int \frac{\lb\sqrt{1-4\g}\rb^{u-s-1}}{\g^{k+1}}\D \g \int \frac{1}{(\A+\B)^{2k-u+s}\eta^{k+1}\A \B}\D \eta.
\end{align*}
Next let us make the change of variable, $\A + \B=\delta$. The image of the $\eta$ curve is a closed contour with $1+t$ in its interior.

We have 
\[
-2\lb \frac{\A+\B}{\A \B}\rb \D \eta  = \D \delta.
\]
Also
\[
\eta = \frac{4\delta^2 t^2 -(\delta^2+t^2-1)^2}{16 \delta^2} = \frac{(1-(\delta-t)^2)((\delta+t)^2-1)}{16 \delta^2}.
\]
Then 
\begin{align*}
    S_2&=-\frac{2^{4k+3}t (t+1)^{2k-u}}{(2\pi\I)^2}\sum\limits_{s=0}^{u}(-1)^{u+s} {u\choose s} (t+1)^s\int \frac{\lb\sqrt{1-4\g}\rb^{u-s-1}}{\g^{k+1}}\D \g \\
    &\times \int \frac{\lb (1-(\delta-t)^2)((\delta+t)^2-1)\rb^{-k-1}}{\delta^{s-1-u}}\D \delta\\
    &=(-1)^{k}\frac{2^{4k+3}t (t+1)^{2k-u}}{(2\pi\I)^2}\sum\limits_{s=0}^{u}(-1)^{u+s} {u\choose s} (t+1)^s\int \frac{\lb\sqrt{1-4\g}\rb^{u-s-1}}{\g^{k+1}}\D \g \\
    \notag &\times \int \frac{\delta^{u+1-s}}{\lb(\delta^2-(t+1)^2)(\delta^2-(t-1)^2)\rb^{k+1}}\D \delta.
\end{align*}
Let us introduce one more change of variable to make the computation easier: 
\[
\delta^2-(t+1)^2=\B.
\]
Then we have 
\begin{align*}
    S_2&=(-1)^{k}\frac{2^{4k+2}t (t+1)^{2k-u}}{(2\pi\I)^2}\sum\limits_{s=0}^{u}(-1)^{u+s} {u\choose s} (t+1)^s\int \frac{\lb\sqrt{1-4\g}\rb^{u-s-1}}{\g^{k+1}}\D \g 
     \int \frac{\lb \B+(t+1)^2\rb^{\frac{u-s}{2}}}{\lb\B(\B+4t)\rb^{k+1}}\D \B.
\end{align*}
Note that the contour in $\B$ variable is a simple closed curve with origin in its interior.
We rewrite (replacing $s$ by $u-s$ in the summation), 
\begin{align*}
    S_2&=(-1)^{k}\frac{2^{4k+2}t (t+1)^{2k}}{(2\pi\I)^2}\sum\limits_{s=0}^{u}(-1)^s {u\choose s} (t+1)^{-s}\int \frac{\lb 1-4\g\rb^{\frac{s-1}{2}}}{\g^{k+1}}\D \g \int \frac{\lb \B+(t+1)^2\rb^{\frac{s}{2}}}{\lb\B(\B+4t)\rb^{k+1}}\D \B.
\end{align*}
We note that only those terms for which $s$ is even survive. 
Therefore we can write $S_2$ as 
\begin{align*}
    S_2&=(-1)^{k}\frac{2^{4k+2}t (t+1)^{2k}}{(2\pi\I)^2}\sum\limits_{s=0, s \mathrm{-even}}^{u} {u\choose s} (t+1)^{-s}\int \frac{\lb 1-4\g\rb^{\frac{s-1}{2}}}{\g^{k+1}}\D \g \int \frac{\lb \B+(t+1)^2\rb^{\frac{s}{2}}}{\lb\B(\B+4t)\rb^{k+1}}\D \B.
\end{align*}
We now assume that $u$ is even. The odd case can be dealt with similarly, and we will not give the proof separately. We have 
\begin{align*}
    S_2&=\frac{(-1)^{k}2^{4k+2}t (t+1)^{2k}}{(2\pi\I)^2}\sum\limits_{m=0}^{u/2}{u \choose 2m}(t+1)^{-2m}\int \frac{\lb 1-4\g\rb^{\frac{2m-1}{2}}}{\g^{k+1}}\D \g  \int \frac{\lb \B+(t+1)^2\rb^{m}}{\lb\B(\B+4t)\rb^{k+1}}\D \B\\
    &=\frac{(-1)^{k}2^{4k+2}t (t+1)^{2k}}{(2\pi\I)^2}\sum\limits_{m=0}^{u/2}{u \choose 2m}(t+1)^{-2m}\int \frac{\lb 1-4\g\rb^{\frac{2m-1}{2}}}{\g^{k+1}}\D \g \int\sum\limits_{q=0}^{m}{m\choose q} \frac{\B^q(t+1)^{2m-2q}}{\lb\B(\B+4t)\rb^{k+1}}\D \B\\
    &=\frac{(-1)^{k}2^{4k+2}t (t+1)^{2k}}{(4t)^{k+1}(2\pi\I)^2}\sum\limits_{m=0}^{u/2}\sum\limits_{q=0}^{m}{u \choose 2m}{m\choose q} (t+1)^{-2q}\int \frac{\lb 1-4\g\rb^{\frac{2m-1}{2}}}{\g^{k+1}}\D \g  \int \frac{1}{\B^{k-q+1} \lb 1+\frac{\B}{4t}\rb^{k+1}}\D \B\\
    &=\frac{(-1)^{k}2^{2k}(t+1)^{2k}}{t^{k}(2\pi\I)^2}\sum\limits_{m=0}^{u/2}\sum\limits_{q=0}^{m}{u \choose 2m}{m\choose q} (t+1)^{-2q}\int \frac{\lb 1-4\g\rb^{\frac{2m-1}{2}}}{\g^{k+1}}\D \g  \int \frac{1}{\B^{k-q+1}} \sum\limits_{p\geq 0} {k+p\choose p} \frac{(-\B)^{p}}{(4t)^{p}} \D \B\\
    &=\frac{(t+1)^{2k-u}}{t^{2k}2\pi\I}\sum\limits_{m=0}^{u/2}\sum\limits_{q=0}^{m}(-1)^q (4t)^{q}(t+1)^{u-2q}{u \choose 2m}{m\choose q}\int \frac{\lb 1-4\g\rb^{\frac{2m-1}{2}}}{\g^{k+1}}\D \g {2k-q\choose k}. 
\end{align*}
We have 
\begin{align*}
   \frac{1}{2\pi \I} \int \frac{\lb 1-4\g\rb^{\frac{2m-1}{2}}}{\g^{k+1}}\D \g=\frac{(-1)^m{2m\choose m} {2k-m\choose k}}{{2k-m\choose m}}.
\end{align*}
Then 
\begin{align*}
    S_2=\frac{(t+1)^{2k-u}}{t^{2k}}\sum\limits_{m=0}^{u/2}\sum\limits_{q=0}^{m}(-1)^{q+m} (4t)^{q}(t+1)^{u-2q}\frac{{u \choose 2m}{m\choose q}{2m\choose m} {2k-m\choose k}{2k-q\choose k}}{{2k-m\choose m}}. 
\end{align*}
Expanding $(t+1)^{u-2q}$, we get, 
\begin{align*}
    S_2=\frac{ (t+1)^{2k-u}}{t^{2k}}\sum\limits_{m=0}^{u/2}\sum\limits_{q=0}^{m}\sum\limits_{r=0}^{u-2q}(-1)^{q+m} (4t)^{q}{u-2q \choose r} t^{r}\frac{{u \choose 2m}{m\choose q}{2m\choose m} {2k-m\choose k}{2k-q\choose k}}{{2k-m\choose m}}.
\end{align*}
We now look at specific coefficients of a fixed power of $t$ inside the summation. With this in mind, let us set $q+r=j$. Note that $0\leq j \leq u$. Then we get the following: The coefficient of $t^j$ in the summation is 
\begin{align*}
    C(j)&:=\sum\limits_{m=0}^{u/2}\sum\limits_{q=0}^{m}\frac{(-1)^{q+m}4^{q}{u-2q\choose j-q}{u \choose 2m}{m\choose q}{2m\choose m} {2k-m\choose k}{2k-q\choose k}}{{2k-m\choose m}}\\
    &=\sum\limits_{q=0}^{u/2} \sum\limits_{m=q}^{u/2}\frac{(-1)^{q+m}4^{q}{u-2q\choose j-q}{u \choose 2m}{m\choose q}{2m\choose m} {2k-m\choose k}{2k-q\choose k}}{{2k-m\choose m}}.
\end{align*}
With this, we have 
\begin{align*}
    S_2=\frac{ (t+1)^{2k-u}}{t^{2k}}\sum\limits_{j=0}^{u} C(j)t^{j}.
\end{align*}
We will simplify the expression for $C(j)$ in the next lemma. Using the simplification, we have the required expression for $S(u)$ given by \eqref{S(u) simplified}. We have introduced terms such as $S, S_1, S_2$ etc. in order to not have to carry over the constants appearing in the summations. For the sake of clarity, we re-emphasize roles of different symbols as follows. We have 
\[
\begin{aligned}
S(u)=S&=\frac{(\lambda(1+t))^{u}k!B_{l,s}}{2^{k} u!t}\sum\limits_{s=0}^{u} 2^{s}{2k-s \choose k-s}{u\choose s} S_1.
\end{aligned}
\]
\[
\begin{aligned}
S_2&=\sum\limits_{s=0}^{u} 2^{s}{2k-s \choose k-s}{u\choose s} S_1\\
&=\frac{ (t+1)^{2k-u}}{t^{2k}}\sum\limits_{j=0}^{u} C(j)t^{j}.
\end{aligned}
\]

\epr
\begin{lemma} The summation $C(j)$ simplifies to 
\[
C(j)=2^{u} {u\choose j} {2k-j \choose k}{2k-u+j \choose k}.
\]
    \end{lemma}
\bpr 
We first make a few straightforward observations about $C(j)$. 
\begin{itemize}
    \item The sum is invariant when $j$ is replaced by $u-j$. Hence it is enough to prove for  $0\leq j\leq u/2$.
    \item The sum is $0$ when $j\geq u+1$.
\end{itemize}
Due to the third combinatorial term, we can replace the lower limit of the summation in $m$ by $0$. We first consider summation in $m$. 
We consider
\begin{align}
    C_1:=\sum\limits_{m=0}^{u/2}\frac{(-1)^{m}{u\choose 2m} {m\choose q} {2m\choose m} {2k-m\choose k}}{{2k-m\choose m}}.
\end{align}
Using $ \frac{{2k-m\choose k}}{{2k-m\choose m}}= \frac{{2k-2m\choose k-m}}{{k\choose m}}$ and simplifying, we have
\begin{align*}
    C_1=\frac{(2k-u)! u!}{k! q! (k-q)!}\sum\limits_{m=0}^{u/2}(-1)^{m} {2k-2m \choose 2k-u}{k-q \choose k-m}.
\end{align*}
Now due to the first combinatorial sum inside the summation, we can replace the upper index of the summation by $k$. Further replacing $k-m$ by $m$, we then get, 
\begin{align}\label{Eqs6.41}
    C_1=\frac{(-1)^k(2k-u)! u!}{k! q! (k-q)!}\sum\limits_{m=0}^{k}(-1)^{m} {2m\choose 2k-u}{k-q \choose m}.
\end{align}
In \eqref{Eqs6.41} above, we can assume the summation in $m$ is till $k-q$. We then get, 
\begin{align*}
   \notag C_1&= \frac{(-1)^k (2k-u)! u!}{k! q! (k-q)!(2\pi\I)}\int \frac{1}{z^{2k-u+1}}(1-(1+z)^2)^{k-q} \D z\\
    \notag &=\frac{(-1)^{q}(2k-u)! u!}{k! q! (k-q)!(2\pi\I)}\int \frac{(z+2)^{k-q}}{z^{k-u+q+1}} \D z \\
    \notag &=\frac{(-1)^q 2^{k-q}(2k-u)! u!}{k! q! (k-q)!(2\pi\I)}\int \sum\limits_{r=0}^{k-q} {k-q \choose r} \frac{z^r}{2^{r} z^{k-u+q+1}} \D z \\
    \notag &= \frac{(-1)^q2^{k-q}(2k-u)! u!}{2^{k-u+q} k! q! (k-q)!} {k-q \choose k-u+q}\\
   &= \frac{(-1)^{q} 2^{u-2q}(2k-u)! u!}{k! q! (k-q)!}{k-q \choose u-2q}.
\end{align*}
With this the summation in $q$ becomes
\begin{align*}
    C(j)&=\frac{ 2^{u}(2k-u)! u!}{(k!)^2}\sum\limits_{q=0}^{u/2} {u-2q \choose j-q}{2k-q \choose k} {k-q \choose u-2q}{k\choose q}\\
    &=\frac{ 2^{u}(2k-u)! u!}{(k!)^2}\sum\limits_{q=0}^{u/2} {u-2q \choose j-q}{2k-q \choose k-q} {k-q \choose u-2q}{k\choose q}\\
    &=\frac{ 2^{u}(2k-u)! u!}{(k!)^2}\sum\limits_{q=0}^{u/2} {u-2q \choose j-q}{2k-q \choose u-2q}{2k-u+q\choose k} {k\choose q}\\
    &=\frac{ 2^{u}(2k-u)! u!}{(k!)^2}\sum\limits_{q=0}^{u/2} {2k-q \choose j-q}{2k-j \choose 2k+q-u}{2k-u+q\choose k} {k\choose q}\\
    &=\frac{ 2^{u}(2k-u)! u!}{(k!)^2}{2k-j\choose k}\sum\limits_{q=0}^{u/2}{2k-q \choose j-q} {k-j \choose u-q-j}{k\choose q}.
\end{align*}
In the equalities above, we have repeatedly used the fact 
\[
{a\choose b}{b\choose c}={a\choose c}{a-c\choose b-c}={a\choose c}{a-c\choose a-b}.
\]
We consider the following summation. Here note that we can let the upper limit of the summation index $q$ to be $k$. This is justified by the fact observed earlier that it is enough to consider $0\leq j\leq u/2$. 
\begin{align*}
    C_2:=\sum\limits_{q=0}^{k} {2k-q \choose j-q}{k-j \choose u-q-j } {k\choose q}.
    \end{align*}
    We have 
    \begin{align*}
        C_2&=\sum\limits_{q=0}^{k}\frac{1}{(2\pi\I)^2}\iint \frac{(1+z)^{2k-q}}{z^{j-q+1}}\frac{(1+w)^{k-j}}{w^{u-q-j+1}} { k\choose q} \D z \D w\\
        &=\frac{1}{(2\pi\I)^2}\iint \frac{(1+z)^{2k}}{z^{j+1}}\frac{(1+w)^{k-j}}{w^{u-j+1}} \lb 1+\frac{zw}{1+z}\rb^{k} \D z \D w\\
        &=\frac{1}{(2\pi\I)^2}\iint \frac{(1+z)^{k}}{z^{j+1}}\frac{(1+w)^{k-j}}{w^{u-j+1}} \lb 1+z(1+w)\rb^{k} \D z \D w\\
        &=\frac{1}{(2\pi\I)^2}\iint \frac{(1+z)^{k}}{z^{j+1}}\frac{(1+w)^{k-j}}{w^{u-j+1}} \sum\limits_{q=0}^{k} {k\choose q} z^{q}(1+w)^{q} \D z \D w\\
        &=\sum\limits_{q=0}^{k}{k\choose q} \frac{1}{(2\pi\I)^2}\iint \frac{(1+z)^{k}}{z^{j-q+1}}\frac{(1+w)^{k+q-j}}{w^{u-j+1}}  \D z \D w\\
        &=\sum\limits_{q=0}^{k} {k\choose q} {k\choose j-q}{k+q-j \choose u-j}\\
        &=\sum\limits_{q=0}^{k} {k\choose q} {k\choose k-j+q}{k+q-j \choose u-j}\\
        &=\sum\limits_{q=0}^{k} {k\choose q} {k\choose u-j}{k-u+j \choose j-q}\\
        &={k\choose u-j} {2k-u+j \choose j}.
    \end{align*}
    Now we have 
    \begin{align*}
        C(j)&=\frac{2^u (2k-u)! u!}{(k!)^2}{2k-j \choose k}{k\choose u-j} {2k-u+j \choose j}\\
        &=\frac{2^u (2k-u)! u!}{(k!)^2} \frac{(2k-j)!}{k!(k-j)!} \frac{k!}{(u-j)!(k-u+j)!}\frac{(2k-u+j)!}{j! (2k-u)!}\\
        &=2^{u} {u\choose j} {2k-j \choose k}{2k-u+j \choose k}.
    \end{align*}
    \epr 
\bibliographystyle{plain}
\bibliography{References}

@article {LVN,
    AUTHOR = {Nguyen, Linh V.},
     TITLE = {Range description for a spherical mean transform on spaces of
              constant curvature},
   JOURNAL = {J. Anal. Math.},
  FJOURNAL = {Journal d'Analyse Math\'{e}matique},
    VOLUME = {128},
      YEAR = {2016},
     PAGES = {191--214},
      ISSN = {0021-7670},
   MRCLASS = {44A15 (42C40 53C65)},
  MRNUMBER = {3481173},
       DOI = {10.1007/s11854-016-0006-z},
       URL = {https://doi.org/10.1007/s11854-016-0006-z},
}

@article {AN,
    AUTHOR = {Agranovsky, Mark and Nguyen, Linh V.},
     TITLE = {Range conditions for a spherical mean transform and global
              extendibility of solutions of the {D}arboux equation},
   JOURNAL = {J. Anal. Math.},
  FJOURNAL = {Journal d'Analyse Math\'{e}matique},
    VOLUME = {112},
      YEAR = {2010},
     PAGES = {351--367},
      ISSN = {0021-7670},
   MRCLASS = {35L20 (35A01 44A15)},
  MRNUMBER = {2763005},
MRREVIEWER = {Chong Kyu Han},
       DOI = {10.1007/s11854-010-0033-0},
       URL = {https://doi.org/10.1007/s11854-010-0033-0},
}

@book {Egorychev,
    AUTHOR = {Egorychev, Georgy P.},
     TITLE = {Integral Representation and the Computation of Combinatorial
              Sums},
    SERIES = {Translations of Mathematical Monographs},
    VOLUME = {59},
      NOTE = {Translated from the Russian by H. H. McFadden,
              Translation edited by Lev J. Leifman},
 PUBLISHER = {American Mathematical Society, Providence, RI},
      YEAR = {1984},
     PAGES = {x+286},
      ISBN = {0-8218-4512-8},
   MRCLASS = {05A15 (32A25)},
  MRNUMBER = {736151},
       DOI = {10.1090/mmono/059},
       URL = {https://doi.org/10.1090/mmono/059},
}

@misc{Riedel_Notes, 
AUTHOR={Riedel, Marko R.}, 
TITLE={Egorychev method and the evaluation of
combinatorial sums},
NOTE={https://pnp.mathematik.uni-stuttgart.de/iadm/Riedel/papers/egorychev.pdf},}

@article {R,
    AUTHOR = {Rubin, Boris},
     TITLE = {Inversion formulae for the spherical mean in odd dimensions
              and the {E}uler-{P}oisson-{D}arboux equation},
   JOURNAL = {Inverse Problems},
  FJOURNAL = {Inverse Problems. An International Journal on the Theory and
              Practice of Inverse Problems, Inverse Methods and Computerized
              Inversion of Data},
    VOLUME = {24},
      YEAR = {2008},
    NUMBER = {2},
     PAGES = {025021, 10},
      ISSN = {0266-5611},
   MRCLASS = {44A12 (26A33 35Q15 35R30)},
  MRNUMBER = {2408558},
MRREVIEWER = {Dmitry G. Shepelsky},
       DOI = {10.1088/0266-5611/24/2/025021},
       URL = {https://doi.org/10.1088/0266-5611/24/2/025021},
}

@article {K,
    AUTHOR = {Kunyansky, Leonid A.},
     TITLE = {Explicit inversion formulae for the spherical mean {R}adon
              transform},
   JOURNAL = {Inverse Problems},
  FJOURNAL = {Inverse Problems. An International Journal on the Theory and
              Practice of Inverse Problems, Inverse Methods and Computerized
              Inversion of Data},
    VOLUME = {23},
      YEAR = {2007},
    NUMBER = {1},
     PAGES = {373--383},
      ISSN = {0266-5611},
   MRCLASS = {44A12 (65R10)},
  MRNUMBER = {2302980},
MRREVIEWER = {Fritz Keinert},
       DOI = {10.1088/0266-5611/23/1/021},
       URL = {https://doi.org/10.1088/0266-5611/23/1/021},
}

@article{finch2006range,
  title={The range of the spherical mean value operator for functions supported in a ball},
  author={Finch, David and Rakesh},
  journal={Inverse Problems},
  volume={22},
  number={3},
  pages={923},
  year={2006},
  publisher={IOP Publishing}
}

@article {Keijo_partial_function,
    AUTHOR = {Ilmavirta, Joonas and M\"{o}nkk\"{o}nen, Keijo},
     TITLE = {Unique continuation of the normal operator of the x-ray
              transform and applications in geophysics},
   JOURNAL = {Inverse Problems},
  FJOURNAL = {Inverse Problems. An International Journal on the Theory and
              Practice of Inverse Problems, Inverse Methods and Computerized
              Inversion of Data},
    VOLUME = {36},
      YEAR = {2020},
    NUMBER = {4},
     PAGES = {045014, 23},
      ISSN = {0266-5611},
   MRCLASS = {44A12 (86A22)},
  MRNUMBER = {4103726},
       DOI = {10.1088/1361-6420/ab6e75},
       URL = {https://doi.org/10.1088/1361-6420/ab6e75},
}

@article {Keijo_partial_vector_field,
    AUTHOR = {Ilmavirta, Joonas and M\"{o}nkk\"{o}nen, Keijo},
     TITLE = {X-ray tomography of one-forms with partial data},
   JOURNAL = {SIAM J. Math. Anal.},
  FJOURNAL = {SIAM Journal on Mathematical Analysis},
    VOLUME = {53},
      YEAR = {2021},
    NUMBER = {3},
     PAGES = {3002--3015},
      ISSN = {0036-1410},
   MRCLASS = {44A12 (46F12 58A10)},
  MRNUMBER = {4261111},
       DOI = {10.1137/20M1344779},
       URL = {https://doi.org/10.1137/20M1344779},
}

@article {Salman_Article,
    AUTHOR = {Salman, Yehonatan},
     TITLE = {Recovering functions from the spherical mean transform with
              limited radii data by expansion into spherical harmonics},
   JOURNAL = {J. Math. Anal. Appl.},
  FJOURNAL = {Journal of Mathematical Analysis and Applications},
    VOLUME = {465},
      YEAR = {2018},
    NUMBER = {1},
     PAGES = {331--347},
      ISSN = {0022-247X,1096-0813},
   MRCLASS = {33C55},
  MRNUMBER = {3806707},
MRREVIEWER = {Tetiana\ A.\ Stepanyuk},
       DOI = {10.1016/j.jmaa.2018.05.019},
       URL = {https://doi.org/10.1016/j.jmaa.2018.05.019},
}

@book {Coddington_Book,
    AUTHOR = {Coddington, Earl A.},
     TITLE = {An Introduction to Ordinary Differential Equations},
    SERIES = {Prentice-Hall Mathematics Series},
 PUBLISHER = {Prentice-Hall, Inc., Englewood Cliffs, N.J.},
      YEAR = {1961},
     PAGES = {xi+292},
   MRCLASS = {34.00},
  MRNUMBER = {0126573},
MRREVIEWER = {L. A. MacColl},
}

@article{ref:AmbKuch,
author = {Gaik Ambartsoumian and Peter Kuchment},
title = {{On the injectivity of the circular Radon transform}},
journal =  {Inverse Problems},
volume = 21,
pages = {473--485}, year = 2005}

@article{ref:AmbKuch-range,
author = {Gaik Ambartsoumian and Peter Kuchment},
title = {{A range description for the planar circular Radon
transform}},
journal =  {SIAM J. Math. Anal.},
volume = 38,
number = 2,
pages = {681--692},
year = 2006}

@article{And,
key = {And},
author = {Lars-Erik Andersson},
title = {{On the determination of a function from
spherical averages}},
journal = {SIAM J.  Math.  Anal.},
volume =  19,
year = 1988,
pages = {214-232},
}

@article {KK,
    AUTHOR = {Kuchment, P. and Kunyansky, L.},
     TITLE = {Half-time range description for the free space wave operator
              and the spherical means transform},
   JOURNAL = {Inverse Problems},
  FJOURNAL = {Inverse Problems. An International Journal on the Theory and
              Practice of Inverse Problems, Inverse Methods and Computerized
              Inversion of Data},
    VOLUME = {41},
      YEAR = {2025},
    NUMBER = {3},
     PAGES = {Paper No. 035005, 24},
      ISSN = {0266-5611,1361-6420},
   MRCLASS = {35L05 (35R30 44A12)},
  MRNUMBER = {4864015},
       DOI = {10.1088/1361-6420/adb0e6},
       URL = {https://doi.org/10.1088/1361-6420/adb0e6},
}

@article{AAKS-even,
      title={A simple range characterization for spherical mean transform in even dimensions}, 
      author={Divyansh Agrawal and Gaik Ambartsoumian and  Venkateswaran Krishnan and Nisha Singhal},
      note={Preprint},
      year={2025},
      eprint={2504.21824},
      journal={arXiv:2504.21824},
      primaryClass={math.CA}
}

@article{AAKN2,
doi = {10.1088/1361-6420/ad8fc8},
url = {https://dx.doi.org/10.1088/1361-6420/ad8fc8},
year = {2024},
publisher = {IOP Publishing},
volume = {40},
number = {12},
pages = {125018},
author = {Agrawal, Divyansh and Ambartsoumian, Gaik and Krishnan, Venkateswaran P and Singhal, Nisha},
title = {On the null space of the backprojection operator and {R}ubin’s conjecture for the spherical mean transform},
journal = {Inverse Problems}
}

@article{anastasio2001comments,
  title={Comments on the filtered backprojection algorithm, range conditions, and the pseudoinverse solution},
  author={Anastasio, Mark A. and Pan, Xiaochuan and Clarkson, Eric},
  journal={IEEE Transactions on Medical Imaging},
  volume={20},
  number={6},
  pages={539--542},
  year={2001},
  publisher={IEEE}
}

@article{clarkson1999projections,
  title={Projections onto the range of the exponential {R}adon transform and reconstruction algorithms},
  author={Clarkson, Eric},
  journal={Inverse Problems},
  volume={15},
  number={2},
  pages={563},
  year={1999},
  publisher={IOP Publishing}
}

@article{mennessier1999attenuation,
  title={Attenuation correction in {SPECT} using consistency conditions for the exponential ray transform},
  author={Mennessier, Catherine  and Noo, Frederic and Clackdoyle, Rolf and Bal, Guillaume and Desbat, Laurent},
  journal={Physics in Medicine \& Biology},
  volume={44},
  number={10},
  pages={2483},
  year={1999},
  publisher={IOP Publishing}
}

@incollection{natterer1983exploiting,
  title={Exploiting the ranges of {R}adon transforms in tomography},
  author={Natterer, Frank},
  booktitle={Numerical Treatment of Inverse Problems in Differential and Integral Equations},
  pages={290--303},
  year={1983},
  publisher={Springer}
}

@article{patch2004thermoacoustic,
  title={Thermoacoustic tomography — consistency conditions and the partial scan problem},
  author={Patch, Sarah K.},
  journal={Physics in Medicine \& Biology},
  volume={49},
  number={11},
  pages={2305},
  year={2004},
  publisher={IOP Publishing}
}

@article{Finch-P-R,
  author =   {David Finch and Sarah Patch and Rakesh},
  title =    {Determining a function from its mean values over
                  a family of spheres},
year =   {2004},
journal = {SIAM J. Math. Anal.},
volume =  {35},
pages = {1213--1240},
}

@article {Seeley,
    AUTHOR = {Seeley, R. T.},
     TITLE = {Spherical harmonics},
   JOURNAL = {Amer. Math. Monthly},
  FJOURNAL = {American Mathematical Monthly},
    VOLUME = {73},
      YEAR = {1966},
    NUMBER = {4, part II},
     PAGES = {115--121},
      ISSN = {0002-9890,1930-0972},
   MRCLASS = {33.27},
  MRNUMBER = {201695},
MRREVIEWER = {Ram\ Kishore\ Saxena},
       DOI = {10.2307/2313760},
       URL = {https://doi.org/10.2307/2313760},
}

@Article{Ambartsoumian2018,
  author   = {Ambartsoumian, Gaik and Gouia-Zarrad, Rim and Krishnan, Venkateswaran P. and Roy, Souvik},
  title    = {Image reconstruction from radially incomplete spherical {R}adon data},
  journal  = {European J. Appl. Math.},
  year     = {2018},
  volume   = {29},
  number   = {3},
  pages    = {470--493},
  issn     = {0956-7925},
  doi      = {10.1017/S0956792517000250},
  mrnumber = {3788452},
}

@book {Trimeche,
    AUTHOR = {Trim\`eche, Khalifa},
     TITLE = {Generalized Harmonic Analysis and Wavelet Packets},
 PUBLISHER = {Gordon and Breach Science Publishers, Amsterdam},
      YEAR = {2001},
     PAGES = {xii+306},
      ISBN = {90-5699-329-1},
   MRCLASS = {43A62 (33C10 42C40 44A15 47B38)},
  MRNUMBER = {1818171},
MRREVIEWER = {Margit R\"{o}sler},
}

@article{Riesz,
AUTHOR={Riesz, Marcel},
TITLE={Int\'{e}grales de {R}iemann-{L}iouville et potentiels},
JOURNAL={Acta Sci. Math. (Szeged)},
VOLUME={9},
NUMBER={1-1},
YEAR={1938-40},
PAGES ={1-42},
}

@article {Kotake-Narasimhan,
    AUTHOR = {Kotake, Takeshi and Narasimhan, Mudumbai S.},
     TITLE = {Regularity theorems for fractional powers of a linear elliptic
              operator},
   JOURNAL = {Bull. Soc. Math. France},
  FJOURNAL = {Bulletin de la Soci\'{e}t\'{e} Math\'{e}matique de France},
    VOLUME = {90},
      YEAR = {1962},
     PAGES = {449--471},
      ISSN = {0037-9484},
   MRCLASS = {47.65 (35.80)},
  MRNUMBER = {149329},
MRREVIEWER = {Mario O. Gonzalez},
       URL = {http://www.numdam.org/item?id=BSMF_1962__90__449_0},
}

@book {CH_Book,
    AUTHOR = {Courant, Richard and Hilbert, David},
     TITLE = {Methods of Mathematical Physics. {V}ol. {II}},
    SERIES = {Wiley Classics Library},
      NOTE = {Partial differential equations,
              Reprint of the 1962 original,
              A Wiley-Interscience Publication},
 PUBLISHER = {John Wiley \& Sons, Inc., New York},
      YEAR = {1989},
     PAGES = {xxii+830},
      ISBN = {0-471-50439-4},
   MRCLASS = {35-00 (00A05 01A75)},
  MRNUMBER = {1013360},
}

@book {Natterer_book,
    AUTHOR = {Natterer, Frank},
     TITLE = {The Mathematics of Computerized Tomography},
    SERIES = {Classics in Applied Mathematics},
    VOLUME = {32},
      NOTE = {Reprint of the 1986 original},
 PUBLISHER = {Society for Industrial and Applied Mathematics (SIAM),
              Philadelphia, PA},
      YEAR = {2001},
     PAGES = {xviii+222},
      ISBN = {0-89871-493-1},
   MRCLASS = {00A69 (44A12 65R10 68U99 92C55)},
  MRNUMBER = {1847845},
MRREVIEWER = {Fritz Keinert},
       DOI = {10.1137/1.9780898719284},
       URL = {https://doi.org/10.1137/1.9780898719284},
}

@article {Agranovsky-Kuchment-Quinto,
    AUTHOR = {Agranovsky, Mark and Kuchment, Peter and Quinto, Eric Todd},
     TITLE = {Range descriptions for the spherical mean {R}adon transform},
   JOURNAL = {J. Funct. Anal.},
  FJOURNAL = {Journal of Functional Analysis},
    VOLUME = {248},
      YEAR = {2007},
    NUMBER = {2},
     PAGES = {344--386},
      ISSN = {0022-1236},
   MRCLASS = {47G10 (43A85 44A12 94A12)},
  MRNUMBER = {2335579},
MRREVIEWER = {Keisaku Kumahara},
       DOI = {10.1016/j.jfa.2007.03.022},
       URL = {https://doi.org/10.1016/j.jfa.2007.03.022},
}

@article {Finch-Haltmeir-Rakesh_even-inversion,
    AUTHOR = {Finch, David and Haltmeier, Markus and Rakesh},
     TITLE = {Inversion of spherical means and the wave equation in even
              dimensions},
   JOURNAL = {SIAM J. Appl. Math.},
  FJOURNAL = {SIAM Journal on Applied Mathematics},
    VOLUME = {68},
      YEAR = {2007},
    NUMBER = {2},
     PAGES = {392--412},
      ISSN = {0036-1399,1095-712X},
   MRCLASS = {35R30 (35L05 65R32 92C55)},
  MRNUMBER = {2366991},
MRREVIEWER = {Alexey\ V.\ Borovskikh},
       DOI = {10.1137/070682137},
       URL = {https://doi.org/10.1137/070682137},
}

@article {Ambartsoumian-Zarrad-Lewis,
    AUTHOR = {Ambartsoumian, Gaik and Gouia-Zarrad, Rim and Lewis, Matthew
              A.},
     TITLE = {Inversion of the circular {R}adon transform on an annulus},
   JOURNAL = {Inverse Problems},
  FJOURNAL = {Inverse Problems. An International Journal on the Theory and
              Practice of Inverse Problems, Inverse Methods and Computerized
              Inversion of Data},
    VOLUME = {26},
      YEAR = {2010},
    NUMBER = {10},
     PAGES = {105015, 11},
      ISSN = {0266-5611,1361-6420},
   MRCLASS = {44A12 (65R99 92C55)},
  MRNUMBER = {2719776},
MRREVIEWER = {H.\ S. P. Shrivastava},
       DOI = {10.1088/0266-5611/26/10/105015},
       URL = {https://doi.org/10.1088/0266-5611/26/10/105015},
}

@article {Agrawal-Krishnan-Sahoo,
    AUTHOR = {Agrawal, Divyansh and Krishnan, Venkateswaran P. and Sahoo,
              Suman Kumar},
     TITLE = {Unique continuation results for certain generalized ray
              transforms of symmetric tensor fields},
   JOURNAL = {J. Geom. Anal.},
  FJOURNAL = {Journal of Geometric Analysis},
    VOLUME = {32},
      YEAR = {2022},
    NUMBER = {10},
     PAGES = {Paper No. 245, 27},
      ISSN = {1050-6926,1559-002X},
   MRCLASS = {53C65 (35J40 45Q05 46F12)},
  MRNUMBER = {4456212},
MRREVIEWER = {Alain\ Brillard},
       DOI = {10.1007/s12220-022-00981-5},
       URL = {https://doi.org/10.1007/s12220-022-00981-5},
}

@article{ilmavirta2023unique,
  title={Unique continuation for the momentum ray transform},
  author={Ilmavirta, Joonas and Kow, Pu-Zhao and Sahoo, Suman Kumar},
  journal={Journal of Fourier Analysis and Applications},
  volume={31},
  number={2},
  pages={17},
  year={2025},
  publisher={Springer}
}

@article {AER,
    AUTHOR = {Antipov, Yuri A. and Estrada, Ricardo and Rubin, Boris},
     TITLE = {Method of analytic continuation for the inverse spherical mean
              transform in constant curvature spaces},
   JOURNAL = {J. Anal. Math.},
  FJOURNAL = {Journal d'Analyse Math\'{e}matique},
    VOLUME = {118},
      YEAR = {2012},
    NUMBER = {2},
     PAGES = {623--656},
      ISSN = {0021-7670,1565-8538},
   MRCLASS = {35R30 (35L20 35Q05 44A12 45Q05)},
  MRNUMBER = {3000693},
MRREVIEWER = {Hideo\ Soga},
       DOI = {10.1007/s11854-012-0046-y},
       URL = {https://doi.org/10.1007/s11854-012-0046-y},
}

@article {Covi-Monkkonen-Railo-UCP,
    AUTHOR = {Covi, Giovanni and M\"{o}nkk\"{o}nen, Keijo and Railo, Jesse},
     TITLE = {Unique continuation property and {P}oincar\'{e} inequality for
              higher order fractional {L}aplacians with applications in
              inverse problems},
   JOURNAL = {Inverse Probl. Imaging},
  FJOURNAL = {Inverse Problems and Imaging},
    VOLUME = {15},
      YEAR = {2021},
    NUMBER = {4},
     PAGES = {641--681},
      ISSN = {1930-8337,1930-8345},
   MRCLASS = {46E35 (35R30)},
  MRNUMBER = {4259671},
MRREVIEWER = {Jiabin\ Zuo},
       DOI = {10.3934/ipi.2021009},
       URL = {https://doi.org/10.3934/ipi.2021009},
}

@book {Krantz-Parks_primer,
    AUTHOR = {Krantz, Steven G. and Parks, Harold R.},
     TITLE = {A Primer of Real Analytic Functions},
    SERIES = {Birkh\"{a}user Advanced Texts: Basler Lehrb\"{u}cher.
              [Birkh\"{a}user Advanced Texts: Basel Textbooks]},
   EDITION = {Second},
 PUBLISHER = {Birkh\"{a}user Boston, Inc., Boston, MA},
      YEAR = {2002},
     PAGES = {xiv+205},
      ISBN = {0-8176-4264-1},
   MRCLASS = {26E05 (26-02 26E10 32B20)},
  MRNUMBER = {1916029},
MRREVIEWER = {Solomon\ Marcus},
       DOI = {10.1007/978-0-8176-8134-0},
       URL = {https://doi.org/10.1007/978-0-8176-8134-0},
}

@book {John-book,
    AUTHOR = {John, Fritz},
     TITLE = {Plane Waves and Spherical Means Applied to Partial
              Differential Equations},
      NOTE = {Reprint of the 1955 original},
 PUBLISHER = {Dover Publications, Inc., Mineola, NY},
      YEAR = {2004},
     PAGES = {iv+172},
      ISBN = {0-486-43804-X},
   MRCLASS = {35-02},
  MRNUMBER = {2098409},
}

@article {Kuchment-Kunyansky-TAT,
    AUTHOR = {Kuchment, Peter and Kunyansky, Leonid},
     TITLE = {Mathematics of thermoacoustic tomography},
   JOURNAL = {European J. Appl. Math.},
  FJOURNAL = {European Journal of Applied Mathematics},
    VOLUME = {19},
      YEAR = {2008},
    NUMBER = {2},
     PAGES = {191--224},
      ISSN = {0956-7925,1469-4425},
   MRCLASS = {92C55 (35L20 44A12 76Q05 78A70 80A99)},
  MRNUMBER = {2400720},
       DOI = {10.1017/S0956792508007353},
       URL = {https://doi.org/10.1017/S0956792508007353},
}

@article {Finch-Rakesh-Survey,
    AUTHOR = {Finch, David and Rakesh},
     TITLE = {The spherical mean value operator with centers on a sphere},
   JOURNAL = {Inverse Problems},
  FJOURNAL = {Inverse Problems. An International Journal on the Theory and
              Practice of Inverse Problems, Inverse Methods and Computerized
              Inversion of Data},
    VOLUME = {23},
      YEAR = {2007},
    NUMBER = {6},
     PAGES = {S37--S49},
      ISSN = {0266-5611,1361-6420},
   MRCLASS = {35L05 (35-02 35C05 35L15)},
  MRNUMBER = {2440997},
MRREVIEWER = {Alexey\ V.\ Borovskikh},
       DOI = {10.1088/0266-5611/23/6/S04},
       URL = {https://doi.org/10.1088/0266-5611/23/6/S04},
}

@article {Cormack-Quinto,
    AUTHOR = {Cormack, Allan M. and Quinto, Eric Todd},
     TITLE = {A {R}adon transform on spheres through the origin in {${\bf
              R}\sp{n}$} and applications to the {D}arboux equation},
   JOURNAL = {Trans. Amer. Math. Soc.},
  FJOURNAL = {Transactions of the American Mathematical Society},
    VOLUME = {260},
      YEAR = {1980},
    NUMBER = {2},
     PAGES = {575--581},
      ISSN = {0002-9947,1088-6850},
   MRCLASS = {44A05 (33A45 35Q05 43A55 58G15)},
  MRNUMBER = {574800},
MRREVIEWER = {R.\ C.\ Varma},
       DOI = {10.2307/1998023},
       URL = {https://doi.org/10.2307/1998023},
}

@article {Rhee,
    AUTHOR = {Rhee, Haewun},
     TITLE = {A representation of the solutions of the {D}arboux equation in
              odd-dimensional spaces},
   JOURNAL = {Trans. Amer. Math. Soc.},
  FJOURNAL = {Transactions of the American Mathematical Society},
    VOLUME = {150},
      YEAR = {1970},
     PAGES = {491--498},
      ISSN = {0002-9947,1088-6850},
   MRCLASS = {35.06},
  MRNUMBER = {262647},
MRREVIEWER = {E.\ J.\ Scott},
       DOI = {10.2307/1995531},
       URL = {https://doi.org/10.2307/1995531},
}

@article {Agranovsky-Finch-Kuchment-range,
    AUTHOR = {Agranovsky, Mark and Finch, David and Kuchment, Peter},
     TITLE = {Range conditions for a spherical mean transform},
   JOURNAL = {Inverse Probl. Imaging},
  FJOURNAL = {Inverse Problems and Imaging},
    VOLUME = {3},
      YEAR = {2009},
    NUMBER = {3},
     PAGES = {373--382},
      ISSN = {1930-8337,1930-8345},
   MRCLASS = {44A12 (92C55)},
  MRNUMBER = {2557910},
MRREVIEWER = {Aleksander\ Denisiuk},
       DOI = {10.3934/ipi.2009.3.373},
       URL = {https://doi.org/10.3934/ipi.2009.3.373},
}

@article {Agranovsky-Kuchment-single_radius,
    AUTHOR = {Agranovsky, Mark and Kuchment, Peter},
     TITLE = {The support theorem for the single radius spherical mean
              transform},
   JOURNAL = {Mem. Differential Equations Math. Phys.},
  FJOURNAL = {Georgian Academy of Sciences. A. Razmadze Mathematical
              Institute. Memoirs on Differential Equations and Mathematical
              Physics},
    VOLUME = {52},
      YEAR = {2011},
     PAGES = {1--16},
      ISSN = {1512-0015},
   MRCLASS = {44A15},
  MRNUMBER = {2883793},
}

@incollection {Tataru-survey,
    AUTHOR = {Tataru, Daniel},
     TITLE = {Unique continuation problems for partial differential
              equations},
 BOOKTITLE = {Geometric Methods in Inverse Problems and {PDE} Control},
    SERIES = {IMA Vol. Math. Appl.},
    VOLUME = {137},
     PAGES = {239--255},
 PUBLISHER = {Springer, New York},
      YEAR = {2004},
      ISBN = {0-387-40529-1},
   MRCLASS = {35B60 (35-02)},
  MRNUMBER = {2169906},
       DOI = {10.1007/978-1-4684-9375-7\_8},
       URL = {https://doi.org/10.1007/978-1-4684-9375-7_8},
}

@article{ambartsoumian2014exterior,
  title={Exterior/interior problem for the circular means transform with applications to intravascular imaging},
  author={Ambartsoumian, Gaik and Kunyansky, Leonid},
  journal={Inverse Problems and Imaging},
  volume={8},
  number={2},
  pages={339--359},
  year={2014},
  publisher={Inverse Problems and Imaging}
}

@article{ambartsoumian2015inversion,
  title={Inversion of a class of circular and elliptical {R}adon transforms},
  author={Ambartsoumian, Gaik and Krishnan, Venkateswaran P.},
  journal={Contemporary Mathematics},
  volume={653},
  year={2015},
  pages={1--12}
}

@article{nguyen2009family,
  title={A family of inversion formulas in thermoacoustic tomography},
  author={Nguyen, Linh V.},
  journal={Inverse Problems and Imaging},
  volume={3},
  number={4},
  pages={649--675},
  year={2009},
  publisher={Inverse Problems and Imaging}
}

@article{norton1980reconstruction,
  title={Reconstruction of a two-dimensional reflecting medium over a circular domain: Exact solution},
  author={Norton, Stephen J.},
  journal={The Journal of the Acoustical Society of America},
  volume={67},
  number={4},
  pages={1266--1273},
  year={1980},
  publisher={Acoustical Society of America}
}

@article{norton1981ultrasonic,
  title={Ultrasonic reflectivity imaging in three dimensions: exact inverse scattering solutions for plane, cylindrical, and spherical apertures},
  author={Norton, Stephen J. and Linzer, Melvin},
  journal={IEEE Transactions on Biomedical Engineering},
  number={2},
  volume={28},
  pages={202--220},
  year={1981},
  publisher={IEEE}
}

@article{aramyan2020recovering,
  title={To recovering the moments from the spherical mean {R}adon transform},
  author={Aramyan, Rafik H. and Mnatsakanov, Robert M.},
  journal={Journal of Mathematical Analysis and Applications},
  volume={490},
  number={2},
  pages={124334},
  year={2020},
  publisher={Elsevier}
}

@article{xu2002time,
  title={Time-domain reconstruction for thermoacoustic tomography in a spherical geometry},
  author={Xu, Minghua and Wang, Lihong},
  journal={IEEE Transactions on Medical Imaging},
  volume={21},
  number={7},
  pages={814--822},
  year={2002},
  publisher={IEEE}
}

@book {Concrete_Mathematics,
    AUTHOR = {Graham, Ronald L. and Knuth, Donald E. and Patashnik, Oren},
     TITLE = {Concrete Mathematics},
   EDITION = {Second},
      NOTE = {A foundation for computer science},
 PUBLISHER = {Addison-Wesley Publishing Company, Reading, MA},
      YEAR = {1994},
     PAGES = {xiv+657},
      ISBN = {0-201-55802-5},
   MRCLASS = {68-01 (00-01 00A05 05-01 68Rxx)},
  MRNUMBER = {1397498},
MRREVIEWER = {Volker\ Strehl},
}

@book {Abramowitz-Stegun,
    AUTHOR = {Abramowitz, Milton and Stegun, Irene A.},
     TITLE = {Handbook of Mathematical Functions with Formulas, Graphs, and
              Mathematical Tables},
    SERIES = {National Bureau of Standards Applied Mathematics Series},
    VOLUME = {No. 55},
 PUBLISHER = {U. S. Government Printing Office, Washington, DC},
      YEAR = {1964},
     PAGES = {xiv+1046},
   MRCLASS = {33.00 (65.05)},
  MRNUMBER = {167642},
MRREVIEWER = {D.\ H.\ Lehmer},
}

@book {Watson,
    AUTHOR = {Watson, George N.},
     TITLE = {A {T}reatise on the {T}heory of {B}essel {F}unctions},
 PUBLISHER = {Cambridge University Press, Cambridge; The Macmillan Company,
              New York},
      YEAR = {1944},
     PAGES = {vi+804},
   MRCLASS = {33.0X},
  MRNUMBER = {10746},
MRREVIEWER = {G.\ Szeg\"{o}},
}

@article{agranovsky1996approximation,
  title={Approximation by spherical waves in ${L}^p$-spaces},
  author={Agranovsky, Mark and Berenstein, Carlos and Kuchment, Peter},
  journal={J. Geom. Anal.},
  volume={6},
  number={3},
  pages={365--383},
  year={1996},
  publisher={Springer}
}

@article{agranovsky1996injectivity,
  title={Injectivity sets for the {R}adon transform over circles and complete systems of radial functions},
  author={Agranovsky, Mark and Quinto, Eric Todd},
  journal={J. Funct. Anal.},
  volume={139},
  number={2},
  pages={383--414},
  year={1996},
  publisher={Elsevier}
}

@book {Aigner,
    AUTHOR = {Aigner, Martin},
     TITLE = {A Course in Enumeration},
    SERIES = {Graduate Texts in Mathematics},
    VOLUME = {238},
 PUBLISHER = {Springer, Berlin},
      YEAR = {2007},
     PAGES = {x+561},
      ISBN = {978-3-540-39032-9},
   MRCLASS = {05-01 (05A05 05A15 05E05)},
  MRNUMBER = {2339282},
MRREVIEWER = {Mikl\'{o}s\ B\'{o}na},
}

@article {ucp-dplane,
    AUTHOR = {Agrawal, Divyansh and Singhal, Nisha},
     TITLE = {{$d$}-plane transform: unique and non-unique continuation},
   JOURNAL = {Proc. Amer. Math. Soc.},
  FJOURNAL = {Proceedings of the American Mathematical Society},
    VOLUME = {153},
      YEAR = {2025},
    NUMBER = {9},
     PAGES = {3841--3853},
      ISSN = {0002-9939,1088-6826},
   MRCLASS = {44A12 (44A35 45Q05)},
  MRNUMBER = {4936338},
MRREVIEWER = {Shubham\ R.\ Jathar},
       DOI = {10.1090/proc/17262},
       URL = {https://doi.org/10.1090/proc/17262},
}

\end{document}